\documentclass[10pt,a4paper,oneside]{amsart}

\usepackage{amsbsy}
\usepackage{xcolor}
\usepackage{graphicx} \usepackage{enumerate} \usepackage{multicol}
\usepackage{mathrsfs} \usepackage[all,cmtip]{xy}
\usepackage{todonotes}
\usepackage[utf8]{inputenc}
\usepackage[english]{babel}
\usepackage{amssymb, amsmath, amstext, amsopn, amsthm, amscd, amsxtra, amsfonts,bbm}
\usepackage{yfonts, mathrsfs}
\usepackage{colonequals}

\usepackage{hyperref}%
\hypersetup{
urlcolor = red,
colorlinks = true,
linkcolor = blue,
citecolor = blue,
linktocpage = true,
pdftitle = {Rough paths InfDim},
pdfauthor = {Erlend Grong, Torstein Nilssen and Alexander Schmeding},
bookmarksopen = true,
bookmarksopenlevel = 1,
unicode = true,
hypertexnames =false
}

\usepackage{cleveref}

\makeatletter
\@namedef{subjclassname@2020}{%
  \textup{2020} Mathematics Subject Classification}
\makeatother

\newcounter{dummy} \numberwithin{dummy}{section}

\newtheorem{la}[dummy]{Lemma}
\newtheorem{prop}[dummy]{Proposition}

\newtheorem{theorem}[dummy]{Theorem}
\newtheorem{corollary}[dummy]{Corollary}
\newtheorem{lemma}[dummy]{Lemma}

\newtheorem{definition}[dummy]{Definition}
\theoremstyle{remark}
\newtheorem{rem}[dummy]{Remark}
\newtheorem{exa}[dummy]{Example}
\newtheorem{remark}[dummy]{Remark}
\newtheorem{example}[dummy]{Example}

\renewcommand{\C}{{\mathbb C}}

\newcommand{\N}{{\mathbb N}}
\newcommand{\R}{{\mathbb R}}

\newcommand{\cA}{{\mathcal A}}
\newcommand{\cI}{{\mathcal I}}

\DeclareMathOperator{\pr}{pr}
\DeclareMathOperator{\rank}{rank}
\DeclareMathOperator{\spn}{span}

\DeclareMathOperator{\so}{\mathfrak{so}}

\newcommand{\bx}{\mathbf{x}}

\newcommand{\ve}{\varepsilon}
\newcommand{\threeN}{| \! | \! |}
\newcommand{\LB}[1][\cdot \hspace{1pt} , \cdot]{\left[\hspace{1pt} #1 \hspace{1pt} \right]}
\newcommand{\Frechet}{Fr\'echet }
\DeclareMathOperator{\Sch}{Sch}
\DeclareMathOperator{\Length}{Length}
\DeclareMathOperator{\GN}{{\it G^{N}}}
\DeclareMathOperator{\Lf}{\mathbf{L}}
\DeclareMathOperator{\LP}{\mathcal{P}}
\DeclareMathOperator{\Pro}{Pr}
\newcommand{\coloneq}{\colonequals}
\newcommand{\bGamma}{\boldsymbol \Gamma}

\DeclareMathOperator{\im}{Im}

\numberwithin{equation}{section}

\title[Geometric rough paths on infinite dimensional spaces]{Geometric rough paths on infinite dimensional spaces}
\author[E.~Grong, Torstein Nilssen and A.~Schmeding]{Erlend Grong, Torstein Nilssen and Alexander Schmeding}

\address{University of Bergen, Department of Mathematics, P.O.~Box 7803, 5020 Bergen, Norway}
\email{erlend.grong@uib.no}
\address{Department of Mathematics, University of Agder, P.O.~Box 422, 4604 Kristiansand, Norway}
\email{torstein.nilssen@uia.no}
\address{FLU, Nord university, Høgskoleveien 27, 7601 Levanger, Norway}
\email{alexander.schmeding@nord.no}



\begin{document}

\begin{abstract}
Similar to ordinary differential equations, rough paths and rough differential equations can be formulated in a Banach space setting. For $\alpha\in (1/3,1/2)$, we give criteria for when we can approximate Banach space-valued weakly geometric $\alpha$-rough paths by signatures of curves of bounded variation, given some tuning of the H\"older parameter. We show that these criteria are satisfied for weakly geometric rough paths on Hilbert spaces. As an application, we obtain Wong-Zakai type result for function space valued martingales using the notion of (unbounded) rough drivers.
\end{abstract}

\maketitle
\tableofcontents

\section{Introduction}

The theory of rough paths was invented by T.\ Lyons in his seminal article \cite{Lyons98} and provides a fresh look at integration and differential equations driven by rough signals. 
A rough path consists of a H\"older continuous path in a vector space together with higher level information satisfying certain algebraic and analytical properties.
The algebraic identities in turn allow one to conveniently formulate a rough path as a path in nilpotent groups of truncated tensor series, cf.\ \cite{FaV06} for a detailed account.
Similar to the well-known theory of ordinary differential equations, it makes sense to formulate rough paths and rough differential equations with values in a Banach space, \cite{CaLaL07}. It is expected that the general theory carries over to this infinite-dimensional setting, yet  a number of results which are elementary cornerstones of rough path theory are still unknown in the Banach setting. 

In \cite{BaR19} the authors introduce the notion of a \emph{rough driver}, which are vector fields with an irregular time-dependence. Rough drivers provides a somewhat generalized description of necessary conditions for the well-posedness of a rough differential equation and the authors use this for the construction of flows generated by these equations. The push-forward of the flow, at least formally, satisfies a (rough) partial differential equation, and this equation is studied rigorously in \cite{BaGu15} where the authors introduce the notion of \emph{unbounded rough drives}. This theory was further developed in \cite{DeGuHoTi16,HH17,HN} in the linear setting (although \cite{DeGuHoTi16} also tackles the kinetic formulation of conservation laws) as well as nonlinear perturbations in \cite{CHLN,HLN18,HLN19,hocquet2018quasilinear,HNS}. Still, the unbounded rough drivers studied in these papers assume a factorization of time and space in the sense that the vector fields lies in the \emph{algebraic tensor} of the time and space dependence. 

Our main motivation for this paper is the observation in \cite{CaN19} that rough drivers can be understood as rough paths taking values in the space of sufficiently smooth functions, see Section \ref{sec:WongZakai}. Moreover, in \cite{CaN19} the authors needed unbounded rough drivers for which the factorization of time and space was \emph{not} valid, and in particular approximating the unbounded rough driver by smooth drivers. In finite dimensions, sufficient conditions that guarantee the existence of smooth approximations can be easily checked and is the so-called \emph{weakly geometric rough paths}. In \cite{CaN19} and ad-hoc method was introduced to tackle the lack of a similar result in infinite dimensions. For other papers dealing with infinite-dimensional rough paths, let us also mention \cite{Dereich10,Bai14,CDLL16}.

In the present paper we address the characterization of weakly geometric rough paths in Banach spaces. Our aims are twofold. Firstly, we describe and develop the infinite-dimensional geometric framework for Banach space-valued rough paths and weakly geometric rough paths. These rough paths take their values in infinite-dimensional groups of truncated tensor products. Some care needs to be taken in this setting, as the tensor product of two Banach spaces will depend on choice of of norm on the product. Secondly, we characterize the geometric rough paths that take their values in an Hilbert space and their relationship to weakly geometric rough paths. Our main result is to prove the following well-known relationship for finite dimensional rough paths in an infinite dimensional setting. Recall that for $\alpha \in (1/3,1/2)$, a geometric $\alpha$-rough path is the is an element of the closure in signatures $S^2(x)_{st} =1 + x_t - x_s + \int_s^t (x_r - x_s) \otimes dx_r$ of curves $x_t$ of bounded variation, while an $\alpha$-rough path $\bx_{st} = 1 + x_{st} + x^{(2)}_{st}$ is called weakly geometric if the symmetric part of $x^{(2)}_{st}$ equals $\frac{1}{2} x_{st} \otimes x_{st}$; a property that holds for all geometric rough paths in particular by an integration by parts argument. Our main result is the following.
\begin{theorem} \label{th:mainHilbert}
For $\alpha \in (1/3, 1/2)$, let $\mathscr{C}^\alpha_g([0,T],E)$ and $\mathscr{C}^\alpha_{wg}([0,T],E)$ denote respectively geometric rough paths and weakly geometric rough paths in a Hilbert space $E$, defined on the interval $[0,T]$ and relative to the Schatten $p$-norm, $1 \leq p \leq \infty$ on $E \otimes E$. Then for any $\beta \in (1/3, \alpha)$, we have inclusions
$$
\mathscr{C}_g^{\alpha}([0,T], E) \subset \mathscr{C}^{\alpha}_{wg}([0,T],E)  \subset \mathscr{C}_g^{\beta}([0,T], E)  .
$$
\end{theorem}
We emphasize that this result includes the Hilbert-Schmidt norm, projective tensor norm and injective tensor norm as respectively $p$ equal to $2$, $1$ and $\infty$.

The structure of the paper is as follows. In Section~\ref{section: Infdim:RP} we review the infinite-dimensional framework for rough paths with values in Banach spaces.
We continue with a presentation of Banach space-valued $\alpha$-rough paths for $\alpha \in (\frac{1}{3}, \frac{1}{2})$ in Section~\ref{section: App:RP}. This leads to the three prerequisite assumptions in Theorem~\ref{thm:geometric rough paths} which states when weakly geometric rough paths can be approximated by signatures of bounded variation path after some tuning of the H\"{o}lder parameter. In Section~\ref{sec:WongZakai}, we apply Theorem~\ref{th:mainHilbert} to prove Wong-Zakai type results for rough flows; a rough generalization of flows of time-dependent vector fields. This yields a concrete application for rough paths on infinite dimensional space.

The remainder of the paper is dedicated to proving Theorem~\ref{th:mainHilbert} by showing that the criteria of Theorem~\ref{thm:geometric rough paths} are indeed satisfied in the Hilbert space setting. All of these criteria depends on considering Carnot-Carath\'eodory geometry or sub-Riemannian geometry of our infinite dimensional groups. Section~\ref{section:CCHilbert} establishes the necessary prerequisite results from finite dimensional Hilbert spaces. We then do the proof of Theorem~\ref{th:mainHilbert} in several steps throughout Section~\ref{sec:hilbert spaces}, including a result in Theorem~\ref{th:main} where we prove that the Carnot-Carath\'eodory metric on the free step~2 nilpotent group generated by a Hilbert space becomes a geodesic distance when restricted to the subset of finite distance from the identity. We conclude the proof of Theorem~\ref{th:mainHilbert} in Section~\ref{sec:ProofHilbert}.

\section{The infinite-dimensional framework for rough paths}\label{section: Infdim:RP}

\subsection{Tensor products of Banach spaces}\label{tensornorm}
If $E$ and $F$ are two Banach spaces, we write $E\otimes_a F$ for their algebraic tensor product. We use the convention that $E^{\otimes_a 0}=\mathbb{R}$. For any $k \geq 0$ we endow the $k$-fold algebraic tensor product $E^{\otimes_a k}$ with a family of norms $\lVert \cdot \rVert_k$ satisfying the following conditions, cf. \cite{BaGaLaY}.
\begin{enumerate}[\rm 1.]
\item  For every $a \in E^{\otimes_a k}, b \in  E^{\otimes_a \ell}$, we have 
$$\lVert a \otimes b\rVert_{k+\ell} \leq \lVert a \rVert_k \cdot \lVert b\rVert_\ell.$$
\item For any permutation $\sigma$ of the integers $1,2 \dots k$ and for any $x_1, \ldots , x_k \in E$,
$$\lVert x_1 \otimes x_2 \otimes \cdots \otimes x_k\rVert_k = \lVert x_{\sigma (1)} \otimes \cdots \otimes x_{\sigma (k)}\rVert_k.$$
\end{enumerate}
Inductively, for $k,\ell \in \N$ we define the spaces $E^{\otimes k} \otimes E^{\otimes \ell}$ as the completion of $E^{\otimes k} \otimes_a E^{\otimes \ell}$
with respect to the norm $\lVert \cdot \rVert_{k+\ell}$. From the inclusions
$$E^{\otimes_a (k+\ell)} \subseteq E^{\otimes k} \otimes_a E^{\otimes \ell} \subseteq E^{\otimes (k+\ell)}$$
it follows that $E^{\otimes k} \otimes E^{\otimes \ell} \cong E^{\otimes (k+\ell)}$ as Banach spaces.

\begin{exa}\label{exa:proj}
The projective tensor product of Banach spaces is the completion of the algebraic tensor product with respect to the projective tensor norm 
$$\textstyle \lVert z \rVert_{\pi} := \inf \left\{ \sum_{i=1}^n \lVert x_i\rVert_E \lVert y_i\rVert_F \, : \,z=\sum_{i=1}^n x_i \otimes y_i \right\}.$$
It is well known that the projective tensor norm satisfies properties 1.\ and 2.\ since it is a reasonable crossnorm on $E \otimes_a F$ (cf.\ \cite[Section 6]{Ryan02}). 
Similarly, the injective tensor norm, defined by
$$\| z \|_\epsilon = \textstyle \sup \left\{ \left|\sum_{i=1}^n \varphi(x_i) \psi(y_i) \right| \, : \, \varphi \in E^*, \psi \in F^*, \| \varphi\| = \|\psi \| = 1, z=\sum_{i=1}^n x_i \otimes y_i \right\} .$$
satisfies 1. and 2. Its completion is the injective tensor product \cite[Section 3]{Ryan02}.

If $E$ is a Hilbert space, then we can identify $E \otimes_a E$ with finite rank operators from $E$ to itself. In this case, the projective and injective norm of $z \colon E \to E$ correspond respectively to the trace norm and the operator norm. Moreover, this identification allows one to identify the projective tensor as the space of nuclear operators $\mathcal{N}(E,E)$ and the injective tensor product as the space of compact operators $\mathcal{K}(E,E)$, see \cite[Corollary 4.8 and Corollary 4.13]{Ryan02} for details.
\end{exa}

\subsection{Algebra of truncated tensor series}\label{subsect:truncalg}
For $N \in \N_0 \cup \{\infty\}$, we define
$$\cA_N := \prod_{k=0}^N E^{\otimes k}$$
as the \emph{the space of (truncated) formal tensor series of $E$}. 
Elements in $\cA_N$ will be denoted as sequences $(x^{(k)})_{k \leq N}$. A sequence concentrated in the $k$-th factor $E^{\otimes k}$ is called \emph{homogeneous of degree $k$}. 
The set $\cA_N$ is an algebra with respect to degree wise addition and the multiplication
$$(x^{(k)})_{k\leq N} \cdot (y^{(k)})_{k\leq N} := \left(\sum_{n+m =k}x^{(n)} \otimes y^{(m)}\right)_{k\leq N}.$$
The algebras $\mathcal{A}_N$ turn out to be Banach algebras for $N$ finite. For $N=\infty$ they are still continuous inverse algebras (CIAs), i.e.~topological algebras such that inversion is continuous and the unit group is an open subset. CIAs and their unit groups can be seen as an infinite-dimensional generalization of matrix algebras and their unit Lie groups. In the case of locally convex spaces more general then Banach spaces (such as $\mathcal{A}_\infty$), we adopt the notion of Bastiani calculus to define smooth maps. This means that we require the existence and continuity of directional derivatives, see \cite{Glo03,keller74} for more information. The relevant results on tensor algebras and their unit groups are summarized in the following result.
\begin{la} \label{lemma:BanachAlgebra}
 The algebra $\cA_N$ is a Banach algebra for $N < \infty$, while $\cA_\infty$ is a \Frechet algebra. Moreover, $\cA_N$ is a continuous inverse algebra whose group of units $\cA_N^\times$ is a $C^0$-regular infinite-dimensional Lie group for any $N \in \N \cup \{\infty\}$.
\end{la}
We recall the notion of regularity of a Lie group $G$.
Let~$1$ denote the group's identity element and $\Lf(G)$ its Lie algebra. Then $G$ is called \emph{$C^r$-regular}, $r\in \N_0\cup\{\infty\}$, if for each $C^r$-curve $u\colon [0,1]\rightarrow \Lf(G)$ the initial value problem 
$$\dot \gamma(t) = \gamma(t) \cdot u(t) \qquad \gamma(0) = 1 $$
 has a (necessarily unique) $C^{r+1}$-solution
 $\text{Evol} (u)\coloneq\gamma\colon [0,1]\rightarrow G$ and the map
 \begin{displaymath}
  \text{evol} \colon C^r([0,1],\Lf(G))\rightarrow G,\quad u\mapsto \text{Evol}(u)(1)
 \end{displaymath}
 is smooth. A $C^\infty$-regular Lie group $G$ is called \emph{regular}
 \emph{(in the sense of Milnor}). Every Banach Lie group is $C^0$-regular (cf. \cite{Neeb06}). Several
 important results in infinite-dimensional Lie theory are only available for regular Lie groups, cf.\ \cite{KM97}.

\begin{proof}[Proof of \Cref{lemma:BanachAlgebra}]
By construction of the algebra structure we have for elements of degree $k$ and $\ell$ that 
\begin{align}\label{eq:mult}
 E^{\otimes k} \cdot E^{\otimes \ell} \subseteq E^{\otimes (k+\ell)}.
\end{align}
By choice of tensor norms in \cref{tensornorm}, $\cA_N$ is a Banach algebra for $N < \infty$, so in particular a continuous inverse algebra. 
Now $\cA_\infty$ is a \Frechet space with respect to the product topology. The choice of tensor norms shows that multiplication is separately continuous and by \cite[VII, Proposition 1]{Waelbroeck} the multiplication is also jointly continuous. Since $\cA_\infty$ is a countable product of Banach spaces whose multiplication satisfies \eqref{eq:mult}, we conclude that $\cA_\infty$ is a densely graded locally convex algebra in the sense of \cite{BaDaS}.
Due to \cite[Lemma B.8 (b)]{BaDaS} $\cA_\infty$ is a continuous inverse algebra, i.e.\ inversion is continuous and the unit group $\cA^\times$ is an open subset of $\cA_\infty$. Following \cite{Glo02,GaN12}, the unit group $\cA^\times_N$ is a regular Banach (for $N< \infty$) or \Frechet Lie group ($N=\infty$). 
\end{proof}

\begin{rem}
 The unit group $\cA^\times_N$ of $\cA_N$ is even a real analytic Lie group in the sense that the group operations extend analytically to the complexification. \end{rem}

\subsection{Exponential map}
Define the canonical projection $\pi_0^N \colon \cA_N \rightarrow \R = \mathcal{A}_0$ and the closed ideal $\cI_{\cA_N} := \ker \pi_0^N = \prod_{0 < k\leq N} E^{\otimes k}$.
Related to this ideal, we introduce the following maps.
\begin{la}[Exponential and logarithm]\label{la:explog}
The exponential and logarithm series
\begin{align*}
 \exp_N \colon \cI_{\cA_N} \rightarrow 1 + \cI_{\cA_N},& \qquad  X \mapsto \sum_{0 \leq n \leq N} \frac{X^{\otimes n}}{n!}, \\
 \log_N \colon 1+ \cI_{\cA_N} \rightarrow \cI_{\cA_N},& \qquad  1 +Y \mapsto \sum_{0 \leq n \leq N} (-1)^{n+1} \frac{Y^{\otimes n}}{n} ,
\end{align*}
yield mutually inverse real analytic isomorphisms.
\end{la}
\begin{proof}

 We follow \cite{Glo02} and define \emph{the spectrum of $x \in \cA_N$} as
 $$\sigma (x) := \C \setminus \{z \in \C \mid z \cdot 1 -x \in \cA_N^\times\},$$
 where $1$ is the unit of $\mathcal{A}_N$.
 For $\Omega \subseteq \C$ open we let $(\cA_N)_\Omega := \{x \in \cA_N \mid \sigma (x) \subseteq \Omega\}$. In view of the holomorphic functional calculus developed in \cite[Section 4]{Glo02} and \cite[Lemma 5.2]{Glo02}, it suffices to prove that $\cI_{\cA_N} \subseteq (\cA_N)_{\{|z|<\log (2)\}}$ and $1+\cI_{\cA_N} \subseteq \{|z-1|<1\}$. 
 However, from \cite[Lemma B.8 (a)]{BaDaS}, the element $z \cdot 1 - x$ is invertible if and only if $z \neq \pi_0^N (x)$. Thus the statement follows from holomorphic functional calculus.
\end{proof}

\begin{rem}\label{rem:BCH}
 Due to \cite[Theorem 5.6]{Glo02} the Lie group exponential of $\cA_N^\times$ is given by the exponential series $$\exp_{\cA_N} \colon \cA_N = \Lf (\cA_N^\times) \rightarrow \cA^\times_N,\quad x \mapsto \sum_{n\in \N_0} \frac{x^{\otimes n}}{n!}.$$
\end{rem}

\subsection{Free nilpotent groups} \label{subsect: Liepoly} Using the exponential map, we are ready to define the subgroups of $\cA_N^\times$ we are interested in. Observe that $\cA_N = \Lf (\cA_N^\times)$ is a Lie algebra with respect to the commutator bracket $ \LB[x,y] := x \otimes y - y \otimes x$. We define inductively the space $\LP_a^n (E)$ of \emph{Lie polynomials} over $E$ of degree $n \in \N$ by $\LP_a^1 (E) := E$ and 
\begin{align*}
\LP_a^{n+1}(E) & := \LP^n_a(E) + \text{span} \{ \LB[x,y] \mid x \in \LP_a^n (E), y \in E\} \subseteq \mathcal{A}_{n+1} ,\\
\LP^\infty_a (E) & := \left\{\sum_{n\in \N_0} P_n \middle| P_n \in E^{\otimes n} \text{ is a Lie polynomial}\right\} .\end{align*}
Elements in the set $\LP^\infty_a (E)$ are called \emph{Lie series}. The set of all Lie polynomials or Lie series is a Lie subalgebra of $(\cA_N , \LB)$, \cite[Chapter 1.2]{Reut93}. Since $\cA_N$ is a topological Lie algebra, we see that also $\LP^N (E) := \overline{\LP_a^N (E)}$ is a closed Lie subalgebra of $(\cA_N , \LB)$. Due to \cite[Theorem 1.4]{Reut93}, we have $\LP^N (E) \subseteq \cI_{\cA_N}.$  
Hence we can apply Lemma \ref{la:explog} and \cite[Corollary 3.3]{Reut93} to see that the set 
$$\GN(E) := \exp_{\cA_N}(\LP^N(E))=\overline{\exp_{\cA_N} (\LP^N(E))},$$
forms a closed subgroup of $\cA_N^\times$.
Closed subgroups of infinite-dimensional Lie groups are in general not again Lie subgroups \cite[Remark IV.3.17]{Neeb06}. So indeed the next proposition is non-trivial. 

\begin{prop} \label{prop:regularity}
The group $\GN(E)$ is a closed submanifold of $\cA_N$ and this structure turns it into a Banach Lie group for $N<\infty$ and into a \Frechet Lie group for $N=\infty$. Moreover, $\GN(E)$ is a $C^0$-regular Lie group and the exponential map $\exp \colon \LP^N(E) \rightarrow \GN(E)$ is a diffeomorphism.
\end{prop}
Observe that for $N< \infty$, the group $\GN(E)$ is a nilpotent group of step $N$ generated by $E$.
\begin{proof}
The group $\GN(E)$ is a closed subgroup of the locally exponential Lie group~$\cA_N^\times$. Due to \Cref{rem:BCH}, the Lie group exponential of this group is $\exp_{\cA_N}$. Define
$$\Lf^{(N)} := \{x \in \cI_{\cA_N} \subseteq \Lf (\cA_N^\times) \mid \exp_{\cA_N} (\R x) \subseteq \GN(E)\}.$$
Due to construction of the closed Lie subalgebra $\LP^N (E)$, we have $\LP^N(E) \subseteq \Lf^{(N)}$.
Conversely as $\LP^N (E) \subseteq \cI_{\cA_N}$ and $\GN(E) \subseteq 1 + \cI_{\cA_N}$, we deduce from \Cref{la:explog} that also $\Lf^{(N)} \subseteq \LP^N (E)$ holds, hence the two sets coincide. It follows that $\GN(E)$ is a locally exponential Lie subgroup of $\cA^\times_N$ by \cite[Theorem IV.3.3]{Neeb06}. 

Since $\LP^N(E) \subseteq \mathcal{I}_{\mathcal{A}_N}$, $\GN(E) \subseteq 1 +\mathcal{I}_{\mathcal{A}_N}$ and the exponential $\exp_{\cA_N}$ is a diffeomorphism between those sets (Lemma \ref{la:explog}), the Lie group exponential induces a diffeomorphism between Lie algebra and Lie group as $\exp =  \exp_{\cA_N}|_{\LP^N(E)}^{\GN(E)}$ due to \cite[Theorem IV.3.3]{Neeb06}.

The Banach Lie groups $\GN(E)$, $N< \infty$ are $C^0$-regular, cf.\ also Remark \ref{re:L1regularity} below, and we see that the canonical projection mappings $\pi^M_N \colon G^{M}(E) \rightarrow \GN(E)$, $N,M \in \N_0 \cup \{\infty\}$, $M \geq N$ are smooth group homomorphisms. Hence we obtain a projective system of Lie groups $(\GN(E) , \pi^{N}_{N-1})_{N \in \N}$ whose Lie algebras also form a projective system $(\LP^N(E),\Lf (\pi^{N}_{N-1}))_{N \in \N}$ of Lie algebras. As sets $$\LP^\infty (E) = \underleftarrow{\lim} \, \LP^N (E), \qquad G^{\infty}(E) = \underleftarrow{\lim} \GN(E),$$ and the limit maps $\pi^\infty_N$ are smooth group homomorphisms. Then we deduce from \cite[Lemma 7.6]{Glo15} that $G^{\infty}(E)$ admits a projective limit chart, hence \cite[Proposition 7.14]{Glo15} shows that $G^{\infty}(E)$ is $C^0$-regular.
\end{proof}

\begin{rem} \label{re:L1regularity}
 The regularity of the Lie groups $\GN(E)$ can be strengthened by weakening the requirements on the curves in the Lie algebra. This results in a notion of $L^p$-regularity \cite{Glo15} for infinite-dimensional Lie groups. One can show that Banach Lie groups such as $\GN(E)$ are $L^1$-regular. Furthermore, as in the proof of \Cref{prop:regularity}, one sees that the limit $G^{\infty}(E)$ is $L^1$-regular. Note that $L^1$-regularity implies all other known types of measurable regularity for Lie groups.
\end{rem}

\begin{example}[Step~$2$] \label{ex:Step2}
For the remainder of the paper, we will mostly focus on the special case of $N =2$. In this case $\LP^2(E)$ is the closure in $\cA_2$ of sums of elements $X$, $Y \wedge Z = Y \otimes Z - Z \otimes Y$ with $X,Y,Z \in E$ and Lie brackets
$$[X + \mathbb{X}, Y + \mathbb{Y}] = X\wedge Y, \qquad X, Y \in E, \mathbb{X}, \mathbb{Y} \in \LP^2(E) \cap E^{\otimes 2}.$$
\end{example}

\section{Applications to infinite dimensional rough paths}\label{section: App:RP}
\subsection{Rough paths and geometric rough paths in Banach space}
Let us first recall the notion of a Banach-space valued rough path, see e.g. \cite{CDLL16}. The definition of a rough path involves higher level components with values in a completed tensor product.

\begin{definition}
Fix $\alpha \in (\frac13, \frac12)$ and a tensor product completion $E \otimes E$ by a choice of a tensornorm $\| \cdot\|_\otimes$ satisfying the assumptions from \Cref{tensornorm}. An $(E,\otimes)$-valued \emph{$\alpha$-rough path} consists of a pair $(x,x^{(2)})$
$$
x\colon  [0,T] \rightarrow E, \qquad x^{(2)} \colon [0,T]^2 \rightarrow E^{\otimes 2} = E\otimes E 
$$
where $x$ is an $\alpha$-H{\"o}lder continuous path and $x^{(2)}$ is ``twice H{\"o}lder continuous'', i.e.
\begin{equation} \label{Holder continuity}
\| x_t - x_s\| \lesssim |t-s|^{\alpha} , \qquad  \| x^{(2)}_{st}\|_2 \lesssim |t-s|^{2 \alpha}.
\end{equation}
In addition, we require
\begin{equation} \label{Chens relation}
x^{(2)}_{st} - x^{(2)}_{su} - x^{(2)}_{ut} = (x_u - x_s) \otimes (x_t - x_u)
\end{equation}
usually called Chen's relation. The set of rough paths equipped with the metric induced by \eqref{Holder continuity} is denoted $\mathscr{C}^{\alpha}([0,T],E)$.
\end{definition}

To be more precise about this distance, we write $\mathbf{x}_{st} = 1+ x_t - x_s + x^{(2)}_{st}$ in~$\mathcal{A}_2$, the two step-truncated tensor algebra over $E$. Chen relation \eqref{Chens relation} can then be rewritten as $\bx_{st} = \bx_{su} \bx_{ut}$. Introduce a metric $d$ on $1 + \mathcal{I}_N  =\{ \bx = 1 + x + x^{(2)} \, : \, x \in E, x^{(2)} \in E \otimes E\}$, by
\begin{align*}
| \bx | & = \max\{ \| x\|, \| x\|^{1/2}_\otimes \}, \\
d(\bx, \mathbf{y}) & = | \bx^{-1} \cdot  \mathbf{y}|  = | (1+ x + x^{(2)} )^{-1} \cdot (1+ y + y^{(2)})|.
\end{align*}
We then define the distance between two $\alpha$-rough paths $(s,t) \mapsto \bx_{st}, \mathbf{y}_{st}$ on $[0,T]^2$ as
\begin{equation} \label{dalpha}
d_\alpha(\bx, \mathbf{y}) = \sup_{0 \leq s < t \leq T} \frac{d(\bx_{st}, \mathbf{y}_{st})}{|t-s|^\alpha}.\end{equation}
Rephrasing these properties, we can define $\bx_t := \bx_{0t} = 1 + x_t + x^{(2)}_{0t} = 1 + x_t + x^{(2)}_{t}$ and regard $t \mapsto \bx_t$ as a $\alpha$-H{\"o}lder continuous path with values in $\mathcal{A}_2$. The relations \eqref{Chens relation} tells us that $\bx_{st} = \bx_s^{-1} \bx_t$
and we have the identification  $\mathscr{C}^{\alpha}([0,T],E) \simeq C^{\alpha}([0,T], 1 + \mathcal{I}_N)$.

If $x_t$ is a smooth path in $E$, then we can lift it to a rough path $\bx_t = 1 + x_t + x_{t}^{(2)}$, where $x_{st}^{(2)} = \int_s^t (x_r - x_s) \otimes dx_r$. Using integration by parts,
\begin{equation} \label{weak geometric}
\int_s^t (x_r - x_s) \otimes dx_r + \int_s^t  dx_r \otimes  (x_r - x_s)  = (x_t - x_s)  \otimes (x_t - x_s),
\end{equation} 
that is, the symmetric part of $x_{st}^{(2)}$ is $(x_t - x_s)  \otimes (x_t - x_s)$. This algebraic condition is equivalent to $\bx_t$ taking values in $G^2(E)$. We note that $\log_2(\bx_{st}) = x_t - x_s + \frac{1}{2} \int_s^t (x_r - x_s)\wedge dx_r$.

\begin{definition}[Weakly geometric and geometric rough paths]
We say that $\alpha$-rough path $\mathbf{x}_t$ is \emph{weakly geometric} if it takes values in $G^2(E)$.
These can again can be given the structure of a metric space $\mathscr{C}^\alpha_{wg}([0,T],E)$ with the metric $d_\alpha$ as in \eqref{dalpha} and can be identified with $C^\alpha([0,T],G^{2}(E))$.

The space of \emph{geometric rough paths} is defined as the closure in the rough path topology of the set canonical lift of smooth paths and is denoted $\mathscr{C}^{\alpha}_g([0,T],E)$.
\end{definition}

Since \eqref{weak geometric} is stable under limits, we get that the set of geometric rough paths can be regarded as a subspace of $C^{\alpha}([0,T],G^{2}(E))$. The reversed question, namely if any $\bx \in C^{\alpha}([0,T],G^{2}(E))$ can be approximated by a sequence of smooth paths is answered positively modulo some tuning of the H\"{o}lder parameter $\alpha$ given the following conditions.

We recall the definition of \emph{the Carnot-Caratheodory metric}, which we will often abbreviate as the CC-metric. We define this metric $\rho$ on $G^2(E)$ by $\rho(\mathbf{y}, \mathbf{z}) = \rho(1, \mathbf{y}^{-1} \cdot \mathbf{z})$ and
$$\rho(1, \mathbf{y}) = \inf \left\{ \int_0^T \| \dot x_t\| \, dt \, : \, \begin{subarray}{c} x \in C([0,T], E), x_0 = 0, \, \text{$x_t$ has bounded variation} \\ \\ \mathbf{y} = S^2(x)_t := 1 + x_T + \int_0^T x_t \otimes dx_t \end{subarray} \right\}. $$

\begin{theorem} \label{thm:geometric rough paths}
Write $$M_{cc}  = \{ \mathbf{z} \in G^2(E) \, : \, \rho(1, \mathbf{z}) < \infty\},$$
and $C([0,T], M_{cc})$ for the space of continuous curves in $M_{cc}$ with respect to $\rho$.

Let $\alpha \in (\frac13, \frac12)$ be given and let $\beta \in (\frac13, \alpha)$ be arbitrary. Assume that the following conditions are satisfied.
\begin{enumerate}[\rm (I)]
\item For some $C >0$ and any $\mathbf{z} \in G^2(E)$, we have $d(1, \mathbf{z}) \leq C\rho(1, \mathbf{z})$.
\item The metric space $(M_{cc}, \rho)$ is a complete, geodesic space.
\item The set
$$C^\alpha([0,T], G^2(E)) \cap C([0,T], M_{cc}),$$
is dense in $C^\alpha([0,T], G^2(E))$ relative to the metric $d_\beta$. 
\end{enumerate}
Then for any $\bx \in C^{\alpha}([0,T],G^{(2)}(E))$ there exists a sequence of bounded variation paths $x^n \colon [0,T] \rightarrow E$ such that
$$\bx^n = S^2(x^n) \rightarrow \bx \text{ in $\mathscr{C}^{\beta}([0,T], E)$.}$$
In particular, we have the inclusions 
$$
\mathscr{C}_g^{\alpha}([0,T], E) \subset C^{\alpha}([0,T],G^{2}(E))  \subset \mathscr{C}_g^{\beta}([0,T], E)  .
$$
\end{theorem}
To explain condition (II) in more details, recall that if $(M, \rho)$ is a metric space, then a curve $\gamma\colon [0,T] \to M$ is said to have \emph{constant speed} if $\Length(\gamma|_{[s,t]}) = c |t- s|$ for any $0 \leq s \leq t \leq T$ and some $c \geq 0$. A constant speed curve is a \emph{geodesic} if $\Length(\gamma|_{[s,t]}) = \rho(\gamma(s), \gamma(t)) = |t - s| \rho(\gamma(0), \gamma(T))$. The metric space $(M, \rho)$ is called \emph{geodesic} if any pair of points can be connected by a geodesic.

If $E$ is finite dimension, the assumptions (I), (II) and (III) hold as $\rho$ and $d$ are then equivalent and we have access to the Hopf-Rinow theorem, see e.g. \cite{FaV06}. If $E$ is a general Hilbert space, the Hopf-Rinow theorem is no longer available \cite{MR540948}. We will also show that the metrics $\rho$ and $d$ will not be equivalent in the infinite dimensional case, yet assumptions (I), (II) and (III) will be satisfied, giving us the result in Theorem~\ref{th:mainHilbert}.
We will prove this statement in Section~\ref{sec:hilbert spaces}, finishing the proof in Section~\ref{sec:ProofHilbert}.

\begin{proof}[Proof of Theorem~\ref{thm:geometric rough paths}]
We first consider the case when $\bx \in C^\alpha([0,T], G^2(E)) \} \cap C([0,T], M_{cc})$.
As $(M_{cc}, \rho)$ is a geodesic space, \cite[Lemma 5.21]{FV10} implies that there exists a sequence of truncated signatures $\bx^{n} = S^2(x^n) \colon [0,T] \rightarrow M_{cc}$ of bounded variation paths $x^n$ such that 
$$
\sup_{t \in [0,T]} \rho( \bx_t, \bx_t^{n}) \rightarrow 0, \qquad \text{for } n \rightarrow \infty,
$$
and we have the uniform bound $\sup_n d( 1,\bx_{st}^n) \leq C|t-s|^{\alpha}$.
From (I), we conclude that $\bx^n$ converges to $\bx$ in $C([0,T],G^{(2)}(E))$. To show the stronger convergence in  $C^{\beta}([0,T],G^{2}(E))$ we perform a classical interpolation argument. Since $d$ is left invariant we see that 
\begin{align*}
d(\bx_{st}^n,\bx_{st}) &  \leq  d( (\bx_s^n)^{-1} \bx_t^n, (\bx_s)^{-1} \bx_t^n)  + d( (\bx_s)^{-1} \bx_t^n, (\bx_s)^{-1} \bx_t)  \\
& \leq 2 \sup_{t \in [0,T]} d(\bx_t^n,\bx_t) \leq 2C \sup_{t \in [0,T]} \rho(\bx_t^n,\bx_t),
\end{align*}
so that there exists a sequence of real numbers $\ve_n \rightarrow 0$ with
$$
d(\bx_{st}^n,\bx_{st}) \leq \ve_n.
$$
From the construction of $\bx^n$ we have $d(1,\bx_{st}^n), d(1,\bx_{st}) \leq C|t-s|^{\alpha}$. Using the interpolation $\min\{a,b\} \leq a^{ \theta}b^{ 1- \theta}$ for every $a,b \geq 0$ and $\theta \in [0,1]$ we have
$$
d(\bx_{st}^n,\bx_{st}) \leq \ve_n \wedge  C|t-s|^{\alpha} \leq \ve^{\theta}_n C^{1 - \theta} |t-s|^{\alpha(1- \theta)}
$$
and by choosing $\theta$ such that $\alpha(1 - \theta) = \beta$ we get convergence
$$
d_\beta(\bx^n, \bx ) =\sup_{s,t \in [0,T]} \frac{ d(\bx_{st}^n, \bx_{st})}{|t-s|^{\beta}} \leq \ve^{\theta}_n C^{1 - \theta} \rightarrow 0, \qquad n \rightarrow \infty.
$$

Finally, from the density of $C^\alpha([0,T], G^2(E)) \cap C([0,T], M_{cc})$ by (III) it follows that if $\bx^m \in C^\alpha([0,T], G^2(E)) \cap C([0,T], M_{cc})$ is a sequence converging to an arbitrary $\bx \in C^\alpha([0,T], G^2(E))$ with respect to $d_\beta$, and $ \bx^{n,m}$ is a sequence of truncated signatures of bounded variation curves converging to $\bx^m$, then $\bx^{m,m}$ converge to~$\bx$. This completes the proof.
\end{proof}

\subsection{Wong-Zakai for stochastic flows} \label{sec:WongZakai}
As an application of Theorem~\ref{thm:geometric rough paths} and Therorem~\ref{th:mainHilbert} we prove a Wong-Zakai type result for martingales with values in a Banach space of sufficiently smooth functions, as systematically explored in \cite{kunita1997stochastic}.
Let $(f_k)_{k=0}^K$ be a collection of  time-dependent vector fields $f_k : [0,T] \times \R^d \rightarrow \R^d$ of class $C^p_b(\R^d,\R^d)$ in the $x$-variable for some $p$ to be determined later, and let $(\omega_t)_{t \in [0,T]}$ be a $K$-dimensional Brownian motion on some filtered probability space $(\Omega, \mathcal{F},\mathbb{P})$. The study of the Stratonovich equation (for notational convenience we write $\omega_t^0 = t$)
\begin{equation} \label{Stratonovich equation}
dy_t = \sum_{k=0}^K f_k(t,y_t) \circ d \omega_t^k 
\end{equation}
is by now classical. 
The book \cite{kunita1997stochastic} stresses the importance of considering the $C^p_b(\R^d,\R^d)$-valued semi-martingale 
\begin{equation} \label{martingale}
m_t(\xi) :=  \sum_{k=0}^K \int_0^t f_k(r,\xi)  d \omega_r^k 
\end{equation}
which allows for a one-to-one characterization of stochastic flows (see \cite{kunita1997stochastic} for precise statement and result). Equation \eqref{Stratonovich equation} is then understood as $dy_t = m_{\circ dt}(y_t)$.

Consider now the tensor product on $C^p_b(\R^d, \R^d)$,
$$(f \otimes g) (\xi,\zeta) := f(\xi) g(\zeta)^T,$$
which allows us to identify $C^p_b(\R^d, \R^d)^{\otimes 2}$ with a subspace of $C^p_b(\R^d \times \R^d, \R^{d \times d})$. Let us define the iterated integral
\begin{align} \label{Stratonovich iteration}
m^{(2)}_{st}(\xi, \zeta) & := \int_s^t (m_r - m_s ) \otimes \circ dm_r(\xi,\zeta) \\ \nonumber
&  := \sum_{k,l=0}^K \int_s^t \int_s^r f_l(v,\xi) f_k(r,\zeta)^T d \omega_v^{l} \circ d\omega_r^k,
\end{align}
as a $C^p_b(\R^d \times \R^d, \R^{d \times d})$-valued random field. Checking the symmetry condition then boils down to checking \eqref{weak geometric} for this tensor product. We have, for $\mu,\nu \in \{1, \dots, d\}$
\begin{align}
&m_{st}^{(2),\mu,\nu}(\xi, \zeta)   + m_{st}^{(2),\nu,\mu}(\zeta, \xi) \notag \\ 
=& \int_s^t (m_r^{\mu}(\xi) - m_s^{\mu}(\xi) ) \circ dm_r^{\nu}(\zeta) + \int_s^t (m_r^{\nu}(\zeta) - m_s^{\nu}(\zeta) ) \circ dm_r^{\mu}(\xi) \notag \\  =&  (m_t^{\mu}(\xi) - m_s^{\mu}(\xi) )  (m_t^{\nu}(\zeta) - m_s^{\nu}(\zeta) ) \label{weakly geometric}
\end{align}
by the well-known integration by parts formula for the Stratonovich integral. We note that the particular decomposition of \eqref{martingale} and \eqref{Stratonovich iteration} in terms of the vector fields $f$ and $\omega$ are not important for this property; only the choice of Stratonovich integration in the definition of $m^{(2)}$ plays a role.

The thread of \cite{kunita1997stochastic} was picked up in the rough path setting in \cite{BaR19} where the authors introduce so-called ``rough drivers'', which are vector field  analogues of rough paths. It was noted in \cite{CaN19} that these vector fields can be canonically defined from infinite-dimensional, i.e. $C^p_b(\R^d,\R^d)$, valued rough paths. In fact, the set of $C^p$-vector fields $\mathfrak{X}^p(\R^d)$ is canonically identified with $C^p_b(\R^d, \R^d)$ via
$$
\begin{array}{ccc}
C^p_b(\R^d,\R^d) & \rightarrow &  \mathfrak{X}^p(\R^d) \\
f & \mapsto & f \cdot \nabla = \sum_{\mu} f^{\mu}  \frac{\partial}{\partial \xi^{\mu}} .\\
\end{array}
$$
Moreover, define by linearity on the algebraic tensor
$$
\begin{array}{ccc}
C^p_b(\R^d, \R^d)^{\otimes_a 2} & \rightarrow &  \mathfrak{X}^p(\R^d) \\
f \otimes g & \mapsto & (f \cdot \nabla (g \cdot \nabla))  = \sum_{\mu, \nu} f^{\mu}  \frac{\partial g^{\nu} }{\partial \xi^{\mu}}  \frac{\partial}{\partial \xi^{\nu}} \\
\end{array}
$$
and denote by $\nabla_2^{\otimes}$ the extension to  $C^p_b(\R^d \times \R^d,\R^{d\times d})$. Moreover, for a matrix $a$ we let $a \nabla^2 := \sum_{\mu,\nu} a^{\mu,\nu} \frac{\partial }{\partial \xi^{\mu}}  \frac{\partial}{\partial \xi^{\nu}}$.
Then, given a rough path $\bx \in \mathscr{C}^{\alpha}([0,T], C^{p}_b(\R^d,\R^d))$, if we let
\begin{equation} \label{rough path to rough driver}
X_{st}(\xi) := x_{st}(\xi) \cdot \nabla  , \qquad  \mathbb{X}_{st}(\xi) := \nabla_2^{\otimes} x_{st}^{(2)}(\xi,\xi) + x_{st}^{(2)}(\xi,\xi) \nabla^2,
\end{equation}
then $\mathbf{X} := (X, \mathbb{X})$ is a weakly geometric rough driver in the sense of \cite{BaR19}. Concretely, we will assume $p \geq 3$ to be an integer for simplicity. This could in principle be relaxed at the expense of introducing vector fields which are Hölder continuous in space, but we stick to the simpler case which is also in line with the regularity assumptions in \cite{FaH14}.

In \cite{BaR19} the authors prove Wong-Zakai approximations of $dy_t = m_{\circ dt}(y_t)$ by using linear interpolation of the Banach-space martingale $m$, showing that the corresponding iterated integral converges to $m^{(2)}$ in the appropriate sense and using continuity of the It\^{o}-Lyons map, see \cite{BaR19} for details. The proposition below is proved in a similar way, except the martingale structure is replaced by Theorem \ref{th:mainHilbert} and the continuity of the mapping $\bx \mapsto \mathbf{X}$. Notice that we use the Sobolev embedding to put ourselves in a Hilbert-space setting.

\begin{theorem}
Let $\bx \in \mathscr{C}_{wg}^{\alpha}([0,T], H^{k}(\R^d,\R^d))$ for $k > \frac{d}{2} + p + 1$ for some $p \geq 3$ and suppose $y$ solves $dy_t = \mathbf{X}_{dt} (y_t)$ where $\mathbf{X}_t = (X_t, \mathbb{X}_t)$ is the rough driver built from $\bx$. Then there exists a sequence of functions $x^n : [0,T] \times \R^d \rightarrow \R^d$ of bounded variation of $t$ such that the solution $y^n$ of
$$
\dot{y}^n_t = x^n_t(y_t^n)
$$
converges to $y$ in $C^{\beta}([0,T], C(\R^d,\R^d))$ for any $\beta \in (\frac13, \alpha)$. 
\end{theorem}
\begin{proof}
Since $\bx$ is weakly geometric we have
\begin{equation} \label{x symmetry}
x_{st}^{(2),\mu,\nu}(\xi, \zeta)   + x_{st}^{(2),\nu,\mu}(\zeta, \xi)   =   (x_t^{\mu}(\xi) - x_s^{\mu}(\xi) )  (x_t^{\nu}(\zeta) - x_s^{\nu}(\zeta) ) 
\end{equation}
for all $\mu,\nu \in \{1, \dots, d\}$ which gives $x^{(2)}_{st} \nabla^2 = \frac12 (x_t - x_{s})(x_t - x_{s})^T \nabla^2$. It follows that 
$$
\mathbb{X}_{st}(\xi) - \frac12 X_{st}(X_{st})(\xi) = \nabla_2^{\otimes} \left( x_{st}^{(2)} - \frac12  x_{st} \otimes x_{st} \right)(\xi,\xi)  \in \mathfrak{X}(\R^d)
$$
so it is a weakly geometric rough driver in the sense \cite{BaR19}. From Theorem \ref{th:mainHilbert} we get can approximate the infinite dimensional rough path $(x,x^{(2)})$ by a sequence of smooth paths. 
The result now follows from \cite[Theorem 2.6]{BaR19} since the embedding $H^{k}(\R^d,\R^d) \subset C_b^p(\R^d, \R^d)$ is continuous. 
\end{proof}

\subsection{Applications to unbounded rough drivers} We briefly mention, at a formal level, how infinite dimensional rough path can be used in the study of rough path partial differential equations. To avoid technicalities we refrain from introducing the full and rather large machinery needed for stating precise results.

Formally, the Lagrangian dynamics $dy_t = \mathbf{X}_{dt}(y_t)$ has a corresponding Eulerian dynamics described by the push-forward $u_t = (y_t)_* \phi$, i.e.
\begin{equation} \label{roughPDE}
du_t = \mathbf{X}_{dt} \nabla u_t , \quad u_0 = \phi.
\end{equation}
The notion \emph{unbounded rough drivers} was introduced in \cite{BaGu15} to give rigorous meaning to \eqref{roughPDE}. In \cite{HN} the notion of \emph{geometric differential rough drivers} was used to characterize a relaxed sufficient condition for the so-called \emph{renormalizability} of unbounded rough drivers. 

Still, the examples where one could verify the condition of a geometric differential rough driver was restricted to the setting when $\mathbf{X}$ belongs to the algebraic tensor of time and space, viz $X_t(x) = \sum_{k=0}^K f_k(x) \omega_t^k$. It was noted in \cite{CaN19} that also unbounded rough drivers can be thought of as infinite dimensional rough paths, see \cite[Section 5]{CaN19} for details. The present paper thus yields approximation results for rough path partial differential equations for more general unbounded rough drivers using \cite[Theorem 2.1]{HN} by simply checking the corresponding symmetry condition \eqref{x symmetry}.

\section{Finite dimensional Carnot-Carath\'eodory geometry} \label{section:CCHilbert}

\subsection{Free nilpotent groups of step 2 in finite dimensions}
Let $E$ be a finite dimensional inner product space and use the notation $X^* = \langle X , \, \cdot \, \rangle$ for any $X \in E$. In the notation of Section~\ref{subsect: Liepoly}, define a Lie algebra $\mathfrak{g}(E) =\LP^2(E)$.  By \Cref{ex:Step2}, we can identify $\mathfrak{g}(E)$ with $ E \oplus \wedge^2 E$ equipped with a Lie bracket structure
\begin{equation} \label{LbracketP2} [X + \mathbb{X}, Y + \mathbb{Y}] =  X \wedge Y, \qquad X,Y \in E, \mathbb{X}, \mathbb{Y} \in \wedge^2 E.\end{equation}
We identify $\wedge^2 E$ with the space of skew-symmetric endomorphisms $\so(E)$ by writing
\begin{equation} \label{identification} X \wedge Y = X^* \otimes Y - Y^* \otimes X. \end{equation}
Consider the corresponding simply connected Lie group $G^{2}(E)$. For the rest of this section, we will use the fact that $\exp: \mathfrak{g}(E) \to G^2(E)$ is a diffeomorphism to identify these as spaces. Using group exponential coordinates $G^2(E)$ is then the space $E \oplus \so(E)$ with multiplication
\begin{equation} \label{product} (x + x^{(2)}) \cdot (y + y^{(2)}) = x+y + x^{(2)} + y^{(2)} + \frac{1}{2} x \wedge y ,\end{equation}
$x,y \in E$, $x^{(2)}, y^{(2)} \in \so(E)$. With this identification the identity is $0$ and inverses are given by $(x+ x^{(2)})^{-1} = -x-x^{(2)}$. Recall the identity in \eqref{exponentialCord} for relating the presentation of $G^2(E)$ as a subset of $\mathcal{A}_2$ and its representation in exponential coordinates.

An absolutely continuous curve $\bGamma(t)$ in $G^2(E)$ with an $L^1$-derivative is called \emph{horizontal} if for almost every $t$,
$$\bGamma(t)^{-1} \cdot \dot \bGamma(t) \in E.$$
In other words, if we write $\bGamma(t) = \gamma(t) + \gamma^{(2)}(t)$ with $\gamma(t) \in E$ and $\gamma^{(2)}(t) \in \wedge^2 E$, then for some $L^1$-function $u(t) \in E$, we have
$$\dot \gamma(t) = u(t), \qquad \dot \gamma^{(2)}(t) = \frac{1}{2} \gamma(t) \wedge u(t).$$

Since $E$ is a generating subspace of $\mathfrak{g}(E)$, it follows from the Chow-Rashevski\"i Theorem \cite{Cho39,Ras38} that any pair of points in $G^2(E)$ can be connected by a horizontal curve. For any pair of points in $\bx, \mathbf{y} \in G^2(E)$, define the Carnot-Carath\'eodory metric (CC-metric) by
$$\rho(\bx, \mathbf{y}) = \left\{ \int_0^1 \| \bGamma(t)^{-1} \cdot \dot \bGamma(t)\|_E \, dt  \, : \, \begin{array}{c} \text{$\bGamma:[0,1] \to G^2(E)$ horizontal,} \\
\bGamma(0) = \bx, \bGamma(1) = \mathbf{y}
\end{array} \right\}.$$
Note that if $\bGamma(t)$ is horizontal, then so is $\bx \cdot \bGamma(t)$. It follows that the distance $\rho$ is left invariant.

From e.g. \cite[Section~7.3]{ABB20},  length minimizers of $\rho$ are all on the form,
\begin{equation} \label{GeodesicFree2} \gamma(t) = x_0 + \int_0^t e^{s\Lambda} u_0 \, ds, \qquad \gamma^{(2)}(t) = x_0^{(2)} + \frac{1}{2} \int_0^t \gamma(s) \wedge e^{s\Lambda} u_0 ds, \end{equation}
for some constant element $\Lambda \in \so(H)$ and $u_0 \in E$.

\begin{example}[Heisenberg group] \label{ex:Heisenberg}
When $E$ is two-dimensional, the group $G^2(E)$ is known as the Heisenberg group. For any choice of orthogonal frame $X$, $Y$, define $Z = \frac{1}{2} (X-iY)$. This means that we can represent any element $\mathbf{y} = aX + b Y + c X \wedge Y$ as
$$\mathbf{y} = (a+ib) Z+ (a-ib) \bar{Z} + c X \wedge Y.$$
We will use a similar notation in the rest of the paper.

If $\Lambda = \lambda X \wedge Y$, $u_0 = u_0 Z + \bar{u}_0 \bar{Z}$, $\gamma(t) = z(t) Z + \bar{z}(t) \bar{Z}$ and $\gamma^{(2)}(t) = \sigma(t) X \wedge Y$ with $z(t), u_0 \in \mathbb{C}$ and $\sigma(t), \lambda \in \mathbb{R}$, then \eqref{GeodesicFree2} becomes
\begin{align*}
z(t) & = z_0 + \int_0^t e^{i\lambda s} u_0 \, ds =  z_0 + \frac{e^{i\lambda t} -1}{i\lambda} u_0 = z_0 + \frac{2\sin(\lambda t/2)}{\lambda} e^{i\lambda /2t} u_0, \\
\sigma(t) & = \sigma_0 + \frac{1}{2} \int_0^t \mathrm{Im}( \bar{z}(s) e^{i \lambda s}u_0 ) \, ds \\
& = \sigma_0 + \frac{2 \sin(\lambda t/2)}{\lambda}  \im\left(  e^{i\lambda/2 t} \bar{z}_0 u_0 \right) - \frac{1}{2} \frac{|u_0|^2}{\lambda}\left( t - \frac{\sin(\lambda t)}{\lambda}  \right).
\end{align*}
where we interpret $\frac{\sin(\lambda t)}{\lambda}$  as $t$ if $\lambda = 0$. If the initial point is the identity $0$, we have
$$z(t)  = \frac{2 \sin(\lambda t/2)}{\lambda} e^{it\lambda/2} u_0, \qquad \sigma(t) = - \frac{|u_0|^2}{2\lambda} \left( t- \frac{\sin (\lambda t)}{\lambda}  \right).$$
If the above geodesic is defined on the intervall $[0,1]$, then it has length $|u_0|$. In particular, we observe the following.
\begin{enumerate}[\rm (a)]
\item A minimizing geodesic defined on $[0,1]$ from $0$ to $z Z + \bar{z} \bar{Z}$ is given by the choice $\lambda =0$. It follows that
$$\rho(0, z Z + \bar{z} \bar{Z}) = |z|.$$
\item A minimizing geodesic defined on $[0,1$] from $0$ to $\sigma X \wedge Y$ is given by the choice $\lambda = \pm 2\pi$ depending on the sign of $\sigma$. Hence, we have that
$$|\sigma| = \frac{\rho(0,  \sigma X \wedge Y)^2}{4\pi}.$$
\item Note that since
$$|z(1)| = |z| = \left| \int_0^1 u(t) dt \right| \leq \int_0^1 |u(t)| dt = |u_0|,$$
we have $|z| \leq \rho(0, z Z + \bar{z} \bar{Z} + \sigma X \wedge Y)$. It then also follows that
\begin{align*}
& 2 \sqrt{\pi} |\sigma|^{1/2}  = \rho(0, \sigma X \wedge Y) \\
& \leq \rho(0, -zZ + \bar{z} \bar{Z}) + \rho(-z Z - \bar{z} \bar{Z}, \sigma X \wedge Y) \\
& = |z| + \rho(0, zZ + \bar{z} \bar{Z} + \sigma X \wedge Y) \leq 2 \rho(0, z Z + \bar{z} \bar{Z}+ \sigma X \wedge Y).
\end{align*}
Using this fact along with the upper bound from the triangle inequality and left invariance, we have
\begin{equation} \label{IneqHeisenberg} \max \{ |z|, \sqrt{\pi} |\sigma|^{1/2}\} \leq \rho(0, zZ + \bar{z} \bar{Z}+ \sigma X \wedge Y) \leq  |z| + 2 \sqrt{\pi} |\sigma|^{1/2}. \end{equation}

\end{enumerate}

\end{example}

\subsection{Dimension-free inequality}\label{subsect: dimensionless}
We want to generalize the inequality \eqref{IneqHeisenberg} to free nilpotent groups of step 2 of arbitrary dimensions. The inequality can be concluded from formulas of the CC-distance to the vertical space in \cite[Appendix A]{RiSe17}, but we include some more details here for the sake of completion and for applications to infinite dimensional vector spaces in Section~~\ref{sec:Infinite}.

Consider the case of a Hilbert space $E$ of arbitrary finite dimension $n\geq 2$.
We want to introduce a class of norms and quasi-norms on $\so(E)$. Any element $\mathbb{X} \in \so(E)$ will have non-zero eigenvalues $\{\pm i \sigma_1, \dots, \pm i \sigma_k\}$ for some $k \geq 0$. We order them in such a way that
$$\sigma_1 \geq \cdots \geq \sigma_k >0.$$
These are also the singular values of $\mathbb{X}$ as $|\mathbb{X} | = \sqrt{- \mathbb{X}^2}$ has exactly these non-zero eigenvalues, with each $\sigma_j$ appearing twice. Define a sequence $\sigma(\mathbb{X}) =(\sigma_j)_{j=1}^\infty$ of non-negative numbers such that $\sigma_j =0$ for $j >k$.
For $0 < p \leq \infty$, we define
$$\| \mathbb{X}\|_{\Sch^p} = 2^{1/p} \| \sigma(\mathbb{X})\|_{\ell^p}.$$
For $p \geq 1$, these are norms called the Schatten $p$-norms \cite[16]{MaV97}. We will also introduce the following map
$$\|\mathbb{X}\|_{cc} = \| \sigma(\mathbb{X})\|_{\ell^1(\mathbb{R};\mathbb{N})} = \sum_{j=1}^\infty j \sigma_j.$$
It is simple to see that $\| \cdot \|_{cc}$ is not a norm when $\dim E > 2$. However, we will show that it is a quasi-norm. Recall that a quasi-norm is a map satisfying the norm axioms except the triangle inequality which is assumed in the form $\|x+y\|\leq K(\|x\|+\|y\|)$ for some $K\geq 1$, \cite[Section I.9]{DaF93}.
From the definition of $\| \cdot \|_{cc}$, we note that
\begin{equation} \label{ComparisonNorm} \frac{1}{2} \| \mathbb{X}\|_{\Sch^1} \leq  \| \mathbb{X} \|_{cc} \leq \frac{1}{4} \|\mathbb{X} \|_{\Sch^{1/2}}.\end{equation}
The latter follows from the fact that for any $k >0$, $\sqrt{a+ kb} \leq \sqrt{a} + \sqrt{b}$ if $b \geq 0$ and $a \geq \frac{(k-1)^2}{4} b$. Hence
$$\sqrt{\sigma_1 + \dots+ k \sigma_k}\leq \sqrt{\sigma_1 + \dots +(k-1) \sigma_{k-1}} + \sqrt{\sigma_k},$$
since $\sigma_1 + \dots +(k-1) \sigma_{k-1} \geq \frac{k(k-1)}{2} \sigma_k$.

Define a homogeneous norm 
$$\threeN x + x^{(2)} \threeN = \max\left\{ \| x \|_E, \sqrt{\pi} \| x^{(2)} \|^{1/2}_{cc} \right\}.$$
We then have the following result.
\begin{theorem} \label{th:Dindep}
Let $E$ be an arbitrary finite dimensional Hilbert space. If $\rho$ is the Carnot-Carath\'eodory distance on $G^2(E)$, then
$$\threeN x + x^{(2)} \threeN  \leq \rho(0, x+ x^{(2)}) \leq 3 \threeN x + x^{(2)} \threeN.$$
\end{theorem}

\begin{proof}
The minimal geodesic from $0$ to $x \in E$ is just a straight line in $E$, and hence
$$\| x\|_E = \rho(0, x).$$
We will show that we also have
\begin{equation} \label{DistanceVert} \rho(0,  x^{(2)}) = 2\sqrt{\pi} \| x^{(2)} \|_{cc}^{1/2}, \qquad x^{(2)} \in \so(E),\end{equation}
The result then follows from similar steps as in Example~\ref{ex:Heisenberg}.

We will use the geodesic equations in \eqref{GeodesicFree2}. Consider a general solution $\bGamma(t) = \gamma(t) + \gamma^{(2)}(t)$ on $G^2(E)$ with $\bGamma(0) = 0$ and $\bGamma(1) = x^{(2)}$. Consider arbitrary initial values $\Lambda \neq 0$ and $u_0 \neq 0$ for the geodesic equation as in \eqref{GeodesicFree2}. Choose an orthonormal basis $X_1, \dots, X_k, Y_1, \dots, Y_k, T_1, \dots, T_{n-k}$ such that we can write
$$\Lambda = \sum_{j=1}^k \lambda_j X_j \wedge Y_j, \qquad \lambda_j > 0.$$
Introduce again complex notation $Z_j = \frac{1}{2}(X_j -i Y_j)$ and write
$$u_0 = \sum_{j=1}^k w_j Z_j + \sum_{j=1}^k \bar{w}_j \bar{Z}_j + \sum_{j=1}^{n-k} c_j T_j, \qquad w_j \in \mathbb{C}, c_j \in \mathbb{R}.$$
We will then have
$$u(t) = \sum_{j=1}^k e^{i\lambda_j t} w_j Z_j + \sum_{j=1}^k e^{-i\lambda_j t} \bar{w}_j \bar{Z}_j  + \sum_{j=1}^{n-k} c_j T_j, \qquad w_j \in \mathbb{C}, c_j \in \mathbb{R}.$$
We make the following simplifications. If $w_j =0$, then the value of $\lambda_j$ has no effect on $u(t)$. We may hence set it to zero and reduce the value of $k$. Without any loss of generality, we can hence assume that every $w_j$ is non-zero. Next, if we have $\lambda_j = \lambda_l$ for some $1\leq j,l \leq k$ then $e^{i\lambda_jt} w_j Z_j + e^{i\lambda_k t} w_k Z_k = e^{i\lambda_jt} (w_j Z_j +  w_k Z_k) =: e^{i\lambda_jt} \frac{w_{jl}}{2} (X_{jl} - iY_{jl})$ for some orthonormal pair of vectors $X_{jl}, Y_{jl}$. Hence we again obtain the same $u(t)$ if we replace $\lambda_j X_j \wedge Y_j + \lambda_l X_l \wedge Y_l$ with $\lambda_j X_{jl} \wedge Y_{jl}$. By repeating such replacements, we may assume that all values of $\lambda_1, \dots, \lambda_k$ are different.

If $\bGamma(t) = \gamma(t) + \gamma^{(2)}(t)$ is the corresponding geodesic, then
\begin{align*}
\gamma(t) = \sum_{j=1}^k \frac{2 \sin(\lambda_j t/2)}{\lambda_j} e^{i\lambda_j t/2} w_j Z_j + \sum_{j=1}^k \frac{2 \sin(\lambda_j t/2)}{\lambda_j} e^{-i\lambda_j t/2} \bar{w}_j \bar{Z}_j  + \sum_{j=1}^{n-k} tc_j T_j. \end{align*}
From the condition $\gamma(1) = 0$, it follows that $c_1, \dots, c_{n-k}$ all vanish for every $1 \leq j \leq n-k$. Furthermore, since we assume that $w_j \neq 0$, it follows that $\lambda_j = 2\pi n_j$ for some positive integers $n_j$. 

Computing $x^{(2)}$ and using that the integers $n_1, \dots, n_k$ are all different, we obtain
\begin{align*}
& x^{(2)}  = \frac{1}{4\pi} \sum_{j=1}^k \mathrm{Im} \left(\frac{|w_j|^2}{i n_j} \int_0^1 (1- e^{2i\pi n_j t}) dt \right) X_j \wedge Y_j  = - \frac{1}{4\pi} \sum_{j=1}^k \frac{|w_j|^2}{n_j}  X_j \wedge Y_j.
\end{align*}
It follows that the endpoint $x^{(2)}$ has $2k$ non-zero eigenvalues $\{ \pm i \sigma_1, \dots, \pm i \sigma_{k}\}$ with
$$\sigma_j = \frac{1}{4n_j \pi} |w_j|^2.$$
In other words, any local length minimizer $\bGamma(t)$ from $0$ to the point $x^{(2)}$ has length
\begin{align*}
\Length(\bGamma)^2 = \sum_{j=1}^k |w_j|^2 & =  \sum_{j=1}^{k} 4 \pi n_j \sigma_j.
\end{align*}
In order to obtain the minimal value, we use $n_j =l$ if $\sigma_j$ is the $l$-th largest eigenvalue. The result follows.
\end{proof}

Using the identity \eqref{DistanceVert} we also obtain the following result.
\begin{corollary}
$\| \, \cdot \, \|_{cc}$ is a quasi-norm on $\so(E)$, even a $1/2$-norm \cite[Section I.9]{DaF93}, in that it satisfies
$$\| \mathbb{X} + \mathbb{Y}\|^{1/2}_{cc} \leq \| \mathbb{X} \|^{1/2}_{cc}+ \| \mathbb{Y}\|^{1/2}_{cc}, \qquad \| \mathbb{X} + \mathbb{Y} \|_{cc} \leq 2(\| \mathbb{X} \|_{cc} +\| \mathbb{Y} \|_{cc}).$$
\end{corollary}

\section{Geometric rough paths on Hilbert spaces} \label{sec:hilbert spaces}

\subsection{Free nilpotent groups on step~2 from Hilbert spaces} \label{sec:Infinite}
Let $E$ be a real Hilbert space, not necessarily finite dimensional. We choose and fix a tensor norm $\|\cdot\|_\otimes$ on the algebraic tensor product $E \otimes_a E$ which is assumed to satisfy properties 1. and 2. from \Cref{tensornorm}. Moreover, we assume that $\|\cdot\|_\otimes$ lies (pointwise) between the injective and projective tensor norms (cf.\ e.g. \cite{Ryan02}). As mentioned in Example~\ref{exa:proj}, we can identity $E\otimes_a E$ with finite rank operators, and we can consider $E \otimes E$ as the closure of finite rank operators with respect to $\| \cdot \|_{\otimes}$.

In describing $\LP^2(E) = E \oplus \bigwedge^2 E$, through \eqref{identification} we identify the algebraic wedge product $\bigwedge_a^2 E$ with the space of all finite rank skew-symmetric operators denoted by $\so_a(E)$. We identify $\LP^2(E)$ with $E \oplus \so_{\otimes}(E)$ where $\so_\otimes(E)$ are the skew-symmetric operators on $E$ that are in the closure of $\so_a(E)$ with respect to $\| \cdot \|_{\otimes}$ and with brackets as in \eqref{LbracketP2}. If we give $\mathfrak{g}(E)$ a norm
$$\|x + x^{(2)}\|_{\mathfrak{g}(E)} = \max \left\{ \|x \|_E,  \| x^{(2)} \|_{\otimes} \right\}.$$
then it has the structure of a Banach Lie algebra.

For any compact skew-symmetric map $\mathbb{X}: E \to E$, define a sequence $\sigma(\mathbb{X}) = (\sigma_j)_{j=1}^\infty$ such that $|\mathbb{X}| = \sqrt{- \mathbb{X}^2}$ has eigenvalues in non-increasing order $\sigma_1 = \sigma_1 \geq \sigma_2 = \sigma_2 \geq \cdots$. 
For $p \in (0, \infty]$, 
let $\so_p(E)$ denote the space of compact skew-symmetric operators $\mathbb{X}$ with finite Schatten $p$-norm $\| \mathbb{X}\|_{\Sch^p} = 2^{1/p} \| \sigma(X)\|_{\ell^p}$. As $p= \infty$ and $p=1$ correspond to respectively the injective norm and the projective norm, we have
$$\so_1(E) \subseteq \so_\otimes(E) \subseteq \so_\infty(E).$$
Introduce the space $\so_{cc}(E)$ as the subspace of $\so_\infty(E)$ of elements $\mathbb{X}$ such that $\| \mathbb{X}\|_{cc} := \sum_{j=1}^\infty j \sigma_j$ is finite. Since all compact operators are limits of finite rank operators (\cite[Corollary 16.4]{MaV97}), all the previously mentioned inequalities from Section \ref{subsect: dimensionless} still hold. In particular, $\| \cdot \|_{cc}$ is a quasi-norm and we have inclusions
$$\so_{1/2}(E) \subseteq \so_{cc}(E) \subseteq \so_1(E).$$

The group $G^2(E)$ corresponding to $\mathfrak{g}(E)$ can be considered in exponential coordinates as the set $\mathfrak{g}(E)$ with group operation as in \eqref{product}. We define the distance $d$ on $G^2(E)$ by $d(\bx, \mathbf{y}) = \| \bx^{-1} \mathbf{y}\|_{\mathfrak{g}(E)}$.
Let $t\mapsto u(t)$ be any function in $L^1([0,1], E)$, and let $\bGamma_{u}$ be the solution of
$$\bGamma_u(t)^{-1} \cdot \dot \bGamma_u(t) = u(t) , \qquad \bGamma_u(0) = 0.$$
Recall that $0$ is the identity, since we are using exponential coordinates. This curve always exists from the $L^1$-regularity property of the Banach Lie group~$G^2(E)$ (see \cite{Glo15} and also \Cref{re:L1regularity}). For any $\mathbf{x}, \mathbf{y} \in G^2(E)$, we define $\rho(\mathbf{x}, \mathbf{y}) \in [0,\infty]$ by
\begin{align*}
\rho(\mathbf{x}, \mathbf{y})& = \rho(0, \mathbf{x}^{-1} \cdot \mathbf{y}), \\
\rho(0, \mathbf{x}) & = \inf \left\{ \| u \|_{L^1} \, : \, u \in L^1([0,1],E), \bGamma_u(1) = \mathbf{x} \right\}.
\end{align*}

\subsection{Properties of projections} \label{sec:Projections}
Let $F$ be a closed subspace of $E$. We write $\Pro_F: E \to F$ for the corresponding orthonormal projection. We then write the map $\pr_F: G^2(E) \to G^2(F)$ for the corresponding map
$$\pr_F(x + x^{(2)}) = \Pro_F + (\Pro_F x^{(2)} \Pro_F), \qquad x \in E, x^{(2)} \in \so_\otimes(E) .$$
We then emphasize the following properties.
\begin{lemma} 
\begin{enumerate}[\rm (a)]
\item $\pr_F$ is a group homomorphism from $G^2(E)$ and $G^2(F)$.
\item Let $\rho_F$ denote the Carnot-Carath\'eodory distance defined on $G^2(F)$. For any $\bx, \mathbf{y} \in G^2(E)$, we have
$$\rho_F(\pr_F \bx, \pr_F \mathbf{y}) = \rho(\pr_F \bx, \pr_F \mathbf{y}) \leq \rho( \bx, \mathbf{y}).$$
In particular, if there is a geodesic from $\bx$ to $\mathbf{y}$ in $F$ with respect to $\rho_F$, then this is also the geodesic in $G^2(E)$ with respect $\rho$.
\end{enumerate}
\end{lemma}
\begin{proof}
(a) follows from the definition of the definition of the group operation. Using (a), we only need to prove that $\rho(0,\pr_F \bx) \leq \rho(0, \bx)$ to prove (b) . We observe that if $\bGamma(t)$ is a horizontal curve from $0$ to $\bx$, then $\pr_F \bGamma(t)$ is a horizontal curve of less or equal length with endpoint $\pr_F \mathbf{x}$.
\end{proof}
\begin{lemma} \label{lemma:ProjSch}
If $\| \cdot \|_{\otimes} = \| \cdot \|_p$ is the Schatten $p$-norm, the following properties hold.
\begin{enumerate}[\rm (a)]
\item For any closed subspace $F$ of $E$ and $\bx, \mathbf{y} \in G^2(E)$, we have
$$d(\pr_F \mathbf{x}, \pr_F \mathbf{y}) \leq d( \mathbf{x},  \mathbf{y}).$$
\item For any $\bx = x + x^{(2)}$, there is a sequence of finite dimensional subspaces $F_1 \subseteq F_2 \subseteq \cdots$ such that
$$x \in F_n \text{ for any $n$,} \qquad \lim_{n \to 0} d(\bx, \pr_{F_n} \bx) = 0.$$
\end{enumerate}
\end{lemma}
\begin{proof}
\begin{enumerate}[\rm (a)]
\item Again it is sufficient to prove that $d(0, \pr_F \bx) \leq d(0, \bf x)$ for any $\bx = x + x^{(2)} \in G^2(E)$. We see that $\| \Pro_F x\|_E \leq \| x \|_E$ and furthermore, $\| \Pro_F x^{(2)} \Pro_F\|_{\Sch^p} \leq \| x^{(2)}\|_{\Sch^p}$ since
\begin{align*}
\sigma_{j+1}(\Pr_F x^{(2)} \Pr_F) & = \max_{\rank(\tilde E) = 2j+1} \min_{\begin{subarray}{c} y \in \tilde E \\ \| y\|_E =1\end{subarray}} \| {\Pr}_{F} x^{(2)} {\Pr}_F y \|_E \\
& \leq  \max_{\rank(\tilde E) = 2j+1} \min_{\begin{subarray}{c} y \in \tilde E \\ \| y\|_E =1\end{subarray}} \|  x^{(2)} y \|_E = \sigma_{j+1}(x^{(2)}).
\end{align*}
\item Write $\bx = x+ x^{(2)}$ and give the singular value decomposition
\begin{equation} \label{SVD} x^{(2)} = \sum_{j=1}^\infty \sigma_j X_j \wedge Y_j, \qquad \sigma_1 \geq \sigma_2 \geq \cdots,\end{equation}
with $X_1, Y_1, X_2, Y_2, \dots $ all orthogonal unit vector fields. We define
$$\tilde F_n = \spn\{ X_j, Y_j \, : \, j =1, \dots, n \}, \qquad   F_n = \spn\{ x, \tilde F_n \} ,$$
Then by left invariance
$$d(\bx, \pr_{F_n} \bx) = d(0, \pr_{F_n^\perp} x^{(2)}) \stackrel{(a)}{\leq} d(0, \pr_{\tilde F_n^\perp} x^{(2)}) = \frac{1}{2} 2^{1/p} \left(\sum_{j=n+1}^\infty \sigma(x^{(2)})^p \right)^{1/p}$$
which converge to zero by definition. \qedhere
\end{enumerate}
\end{proof}

\subsection{Geodesic completeness}
One of the main steps in completing Theorem~\ref{th:mainHilbert} will be to establish that $\rho$ makes a subset into a geodesic space.
\begin{theorem} \label{th:main}
Let $E$ be a Hilbert space and define $G^2(E)$ relative to a tensor norm $\|\cdot\|_\otimes$ satisfying 1. and 2. from \Cref{tensornorm} and bounded from below by the injective tensor product $\| \cdot \|_{\Sch^\infty}$. If we define $G^2(E) := \exp(E \oplus \so_{cc}(E))$, then
$$G_{cc}^2(E) = \{ \mathbf{x} \in G^2(E) \, : \, \rho(0, \mathbf{x}) < \infty \},$$
Furthermore, the metric space $(G_{cc}^2(E),\rho)$ is a complete, geodesic space and if we define
$$\threeN x + x^{(2)} \threeN = \max \left\{ \| x\|_E, \sqrt{\pi} \|x^{(2)} \|_{cc}^{1/2} \right\},$$
then
\begin{equation} \label{Bounds}
\threeN \mathbf{x} \threeN \leq \rho(0,\mathbf{x}) \leq 3 \threeN \mathbf{x} \threeN.
\end{equation}
\end{theorem}

We will do the proof of this theorem in two parts. In the first part, we will show that $G_{cc}^2(E)$ is indeed exactly the set with finite $\rho$-distance and that the inequality \eqref{Bounds} holds. In the second part, we show that it is a geodesic space. 

\begin{proof}[Proof of Theorem~\ref{th:main}, Part I]
We will begin by introducing the following notation, which we will use in both parts of the proof. Recall the definition of $\pr_F: G^2(E) \to G^2(F) \subseteq G^2(E)$ for some closed subspace $F$ from Section~\ref{sec:Projections}. Write $\pr_{F,\perp} = \pr_{F^\perp}$  and write a projection operator
$$\pr_{F \wedge F^\perp}(\bx) ={\Pro}_F x^{(2)} {\Pro}_{F^\perp} + {\Pro}_{F^\perp} x^{(2)} {\Pro}_{F} = \bx - \pr_F \bx - \pr_{F, \perp} \bx, \qquad \bx = x+ x^{(2)}.$$
We have already shown the result for finite dimensional spaces, so we assume that $E$ is infinite dimensional. \smallskip

\paragraph{\it Step 1: The CC-distance is finite on algebraic elements} Let $G_a(E) = \exp(E \oplus \so_a(E)$ and
consider an arbitrary element $\mathbf{x} = x+ x^{(2)} \in G_a^2(E)$ with $x^{(2)} = \sum_{j=1}^n \sigma_j X_j \wedge Y_j$ being the singular value decomposition as in \eqref{SVD}.
Define the finite dimensional subspace $F = \spn \{ x, X_1, Y_1, \dots, X_n, Y_n\}$. We then observe that since $\mathbf{x} \in G^2(F)$, $\rho(0, \mathbf{x}) < \infty$ and there is a minimizing geodesic from $0$ to $\mathbf{x}$. Also, any element in $G_a^2(E)$ satisfies the inequality \eqref{Bounds}.
\smallskip

\paragraph{\it Step 2: Vertical elements} Consider an element $\mathbf{x} = x^{(2)} \in \so_{cc}(E)$ with $\sigma(x^{(2)} ) = (\sigma_j)$. Let $x^{(2)} = \sum_{j=1}^\infty \sigma_j X_j \wedge Y_j$ be the singular value decomposition and define $Z_j = \frac{1}{2} (X_j - Y_j)$. Consider the curve
$$u(t) = 2 \sqrt{\pi} \sum_{j=1}^\infty (j \sigma_j)^{1/2} (e^{-2\pi j t} Z_j + e^{2\pi j t} \bar{Z}_j) .$$
We see that $\|u(t)\|_E = \|u\|_{L^1} =2 \sqrt{\pi \| x^{(2)}\|_{cc}}$. Furthermore, if $F_{n}$ is the span of $X_1$, $Y_1$,  $\dots$, $X_n$, $Y_n$, then by the proof of Theorem~\ref{th:Dindep}, it follows that $\pr_{F_n} \bGamma_{u}$ is a minimizing geodesic from $0$ to $\mathbf{x}^n  := \sum_{j=1}^n \sigma_j X_j \wedge Y_j$.
Since $\pr_{F_n} u$ converges to $u$ in $L^1([0,1],E)$ and $\mathbf{x}^n$ converges to $\mathbf{x}$ in the norm $\| \, \cdot \, \|_{\mathfrak{g}(E)}$, it follows that $\bGamma_u$ is a minimizing geodesic from $0$ to $\mathbf{x}$, and in particular,
$$\rho(0, \mathbf{x}) = \Length(\bGamma_u) =2\sqrt{\pi}\| x^{(2)}\|^{1/2}_{cc}.$$
\smallskip

\paragraph{\it Step 3: The CC-distance is exactly finite on $G_{cc}^2(E)$}
For any element $\mathbf{x}  = x + x^{(2)}\in G_{cc}^2(E)$, we can construct a horizontal curve $\bGamma$ from $0$ to $\mathbf{x}$ by a concatenation of the straight line from $0$ to $x$ with a minimizing geodesic from $0$ to $x^{(2)}$ left translated by $x$. The result is that
$$\rho(0, \mathbf{x})\leq \Length(\bGamma) = \| x\| + 2\sqrt{\pi} \| x^{(2)}\|_{cc}^{1/2} \leq 3 \threeN \mathbf{x} \threeN < \infty.$$
Conversely if $\mathbf{x} \in G^2(E)$ and $\threeN \mathbf{x} \threeN = \infty$, then using \eqref{Bounds} and any sequence $\mathbf{x}^n$ in $G_a^2(E)$ converging to $\mathbf{x}$ in $\| \, \cdot \, \|_{\mathfrak{g}(E)}$, we see that $\rho(0, \mathbf{x}) = \infty$. $G^2_{cc}(E)$ is complete with the distance $\rho$ as it is complete with respect to $\threeN  \cdot  \threeN$ by definition.
\end{proof}

In order for us to complete Part~II of the proof of Theorem~\ref{th:main} and show that $(G_{cc}(E), \rho)$ is a geodesic space, we will need the following lemma.

\begin{lemma} \label{lemma:Compact}
Let $\mathbf{x} = x + x^{(2)} \in G_{cc}^2(E)$ be a fixed arbitrary element with singular value decomposition $x^{(2)} = \sum_{j=1}^\infty \sigma X_j \wedge Y_j$ as in \eqref{SVD}. Define subspaces $F_0 \subseteq F_1 \subseteq F_2\subseteq \cdots$ by
\begin{equation} \label{Fn} F_0 = \spn\{ x\}, \qquad F_{n+1} = \spn (F_n \cup\{ X_{n+1}, Y_{n+1}\}).\end{equation}
For any $n \geq 0$, define $\pr_n = \pr_{F_n}$, $\pr_{n,\perp} = \pr_{F_n^\perp}$ and $\pr_{n,\wedge} = \pr_{F_n \wedge F_n^\perp}$
\begin{enumerate}[\rm (a)]
\item The set
\begin{equation} \label{Kset}
K(\mathbf{x}) = \left\{ \mathbf{y} \in G_{cc}^2(E) \, : \, \begin{array}{c} \text{For any $n \geq 0$,} \\ \\
\rho(0, \pr_n \mathbf{y}) \leq \sqrt{2 \rho(0, \mathbf{x}) \rho(0, \pr_n \mathbf{x})}  \\ \\
\rho(0, \pr_{n,\perp} \mathbf{y}) \leq \sqrt{2 \rho(0, \mathbf{x}) \rho(0, \pr_{n,\perp} \mathbf{x})}  \\ \\
 \| \pr_{n, \wedge} \mathbf{y}\|_{\otimes}  \leq 4\sqrt{2 \rho(0, \mathbf{x})^3 \rho(0, \pr_{n,\perp} \mathbf{x})} \end{array}\right\},\end{equation}
is relatively compact in $G^2(E)$.
\item Any minimizing geodesic from $0$ to $\mathbf{x}$ is contained in $K(\mathbf{x})$.
\end{enumerate}
\end{lemma}

\begin{proof}
To simplify notation in the proof, we write $\rho(\mathbf{x}) := \rho(0,\mathbf{x})$.
\begin{enumerate}[\rm (a)]
\item Define $F_\infty = \spn \{ x, X_1, Y_1, X_2, Y_2, \dots \}$. From the definition of $K(\bf x)$ it follows that $\pr_{F_\infty} \mathbf{y} = 0$ for any $\mathbf{y} \in K(\bx)$. Considering the limit of $\pr_n$, we also have that
$$\|\mathbf{y}\|_{\mathfrak{g}(E)} \leq \rho(\mathbf{y}) \leq \sqrt{2} \rho( \mathbf{x})$$
so $K(\mathbf{x})$ is bounded in both $G^2_{cc}(E)$ and $G^2(E)$. Recall (e.g.\ from \cite[Theorem 4.3.29]{Eng77}) that for a complete metric space, a set is relatively compact if and only if it is totally bounded, i.e. for every $\ve >0$ there is a finite set of balls of radius $\ve >0$ covering the set. Let $B(\mathbf{z}, r)$ be the ball of radius $r$ centered at $\mathbf{z} \in \mathfrak{g}(E)$ with respect to the $\| \cdot \|_{\mathfrak{g}(E)}$-norm. We observe that for any $\mathbf{y} \in K(\mathbf{x})$, we have
\begin{align*}
\| \pr_n \mathbf{y}\|_{\mathfrak{g}(E)} \leq \rho( \pr_n \mathbf{y}) &\leq \sqrt{2} \rho( \mathbf{x}),\\
\|\pr_{n,\perp} \mathbf{y}\|_{\mathfrak{g}(E)} \leq \rho(\pr_{n,\perp} \mathbf{y}) & \leq \sqrt{2 \rho(\mathbf{x}) \rho( \pr_{n,\perp} \, \mathbf{x})} ,\\
\|\pr_{n,\wedge} \mathbf{y}\|_{\mathfrak{g}(E)} = \frac{1}{2} \| \pr_{n, \wedge} \mathbf{y}_3\|_{\otimes} &\leq 4\sqrt{2 \rho(\mathbf{x})^3 \rho(\pr_{n,\perp} \, \mathbf{x})} .
\end{align*}
Hence, for a given $\ve >0$, we can choose $n$ sufficiently large such that
$$\max \left\{\sqrt{2 \rho(\mathbf{x}) \rho( \pr_{n,\perp} \, \mathbf{x})}, 2\sqrt{2 \rho(\mathbf{x})^3 \rho(\pr_{n,\perp} \, \mathbf{x})} \right\} \leq \frac{\ve}{3}.$$
Choose a finite set of points $\mathbf{z}_1, \dots, \mathbf{z}_N$ such that $\cup_{j=1}^N B(\mathbf{z}_j, \frac{\ve}{3})$ covers the relatively compact set $F_{n} \cap B(0,\rho(\mathbf{x}))$. By our choice of $n$, we then have that $\cup_{j=1}^N B(\mathbf{z}_j, \ve)$ covers all of $K(\bx)$.
\item
We observe first that since $\pr_{n,\perp} \bx$ is in the center of $G^2(E)$, then by left invariance
\begin{align*}
\rho(\bx) &= \rho((\pr_{n} \bx) \cdot (\pr_{n,\perp} \bx)) \leq \rho( \pr_{n} \bx ) + \rho( \pr_{n,\perp} \bx)
\end{align*}
Let $\bGamma = \bGamma_{u} = \gamma + \gamma^{(2)}\colon [0,1] \to G^2(E)$ be any minimizing geodesic with left logarithmic derivative $u$ and write $u_n = \pr_n u$ and $u_{\perp,n} = \pr_{n,\perp} u$. Since $u$ is a minimizing geodesic, then by reparametrization, we may assume that 
\begin{align}
\|u(t)\|_E &= \sqrt{\| u_n (t)\|^2_E +\| u_{n,\perp} (t)\|^2_E} = \rho( \mathbf{x}) , \qquad \text{and note } \label{eq:projgeod1}\\
\rho(\pr_{n,\perp} \, \mathbf{x})& \leq \Length(\pr_{n,\perp} \, \bGamma) = \int_0^1 \| u_{n,\perp}(t)\|_E \, dt \leq \rho(\mathbf{x}). \label{eq:projgeod2}
\end{align}
This leads to the following sequence of inequalities
\begin{align*}
& \rho( \mathbf{x}) \rho( \pr_n \mathbf{x}) \geq \rho(\mathbf{x}) \left(\rho(\mathbf{x}) - \rho(\pr_{n,\perp} \mathbf{x}) \right) \\
&\hspace{-.45cm} \stackrel{\eqref{eq:projgeod1} + \eqref{eq:projgeod2}}{\geq} \rho(\mathbf{x}) \int_0^1\left(  \sqrt{\| u_n (t)\|^2_E +\| u_{n,\perp} (t)\|^2_E} - \| u_{n,\perp}(t)  \|_E \right)\, dt \\
& \stackrel{\eqref{eq:projgeod1}}{=} \rho(\mathbf{x}) \int_0^1\left(  \frac{\| u_n(t)\|^2_E}{\sqrt{\| u_n (t)\|^2_E +\| u_{n,\perp} (t)\|^2_E} + \| u_{n,\perp}(t) \|_E} \right)\, dt \\
& \stackrel{\eqref{eq:projgeod2}}{\geq} \frac{1}{2}  \int_0^1\| u_n(t)\|^2_E \, dt \stackrel{\text{Jensen}}{\geq} \frac{1}{2} \left(\int_0^1\| u_n(t)\|_E \, dt \right)^2 = \frac{1}{2} \Length(\pr_n \bGamma)^2.
\end{align*}
It follows that any point $\mathbf{y}$ on the curve $\bGamma$ will have $\rho(\pr_n \mathbf{y}) \leq \sqrt{2\rho(\mathbf{x}) \rho( \pr_n \mathbf{x})}$. By a similar calculation, we have that $\rho(\pr_{n,\perp} \mathbf{y}) \leq \sqrt{2\rho(\mathbf{x}) \rho( \pr_{n,\perp} \mathbf{x})}$.

We also see that $\pr_{n,\wedge} \bGamma(t) =  \pr_{n,\wedge} \gamma^{(2)}(t)$ with
$$\pr_{n,\wedge} \gamma^{(2)}(t) = \frac{1}{2} \int_{0}^t ( (\pr_n \gamma (s)) \wedge u_{n,\perp}(s) +   (\pr_{n,\perp} \gamma(s)) \wedge u_n(s)) ds.$$
We note that $\int_0^t \| u_n(s)\| \, dt \leq \Length(\pr_n \bGamma)$, while $\| \pr_n  \gamma(t) \| \leq \rho(\pr_n \gamma(t)) \leq \Length(\pr_n \bGamma)$. Since we have similar relation applying $\pr_n$, 
We finally use that 
\begin{align*}
& \| \pr_{\wedge} \gamma^{(2)}(t)\|_{\otimes} \leq \Length(\pr_n \bGamma(t)) \Length(\pr_{n,\perp} \bGamma(t)) 
\leq \sqrt{2\rho(\mathbf{x})^3 \rho(\pr_{n, \perp} \mathbf{x})} .
\end{align*}
Hence the geodesic satisfies the pointwise bounds from the definition of $K(\mathbf{x})$ and the result follows. \qedhere
\end{enumerate}
\end{proof}

\begin{lemma} \label{lemma:closed}
For any $r > 0$ and $\bx_0 \in G^2(E)$, we have that the set
$$\bar{B}_\rho(\bx_0, r) = \{ \bx \, : \, \rho(\bx_0, \bx) \leq r \}$$
is closed in $G^2(E)$.
\end{lemma}
\begin{proof}
By left invariance, we only consider $\mathbf{x}_0=0$. Assume that $\mathbf{y}^n= y^n + y^{n,(2)}$ is a sequence contained in $\bar{B}_\rho(0, r)$ converging in $G^2(E)$ to some element $\mathbf{y} = y + y^{(2)}$. We then have that
$$\| y - y^n \|_E \to 0, \qquad \| y^{(2)} - y^{n,(2)} \|_{\otimes} \to 0.$$
In particular, we will have $\| y^{(2)} - y^{n,(2)} \|_{\Sch^\infty} \to 0$ implying that if $\sigma(y^{(2)}) = (\sigma_j)$ and $\sigma(y^{n,(2)}) = (\sigma_j^n)$, then $\sigma_{j}^n \to \sigma_j$. It follows that
\begin{align*}
\threeN \mathbf{y} \threeN \leq \|y \|_E + \sqrt{\pi} \sum_{j=1}^\infty j \sigma_j = \lim_{n\to \infty} \| y^n\|_E + \sqrt{\pi} \lim_{m \to \infty} \left(\lim_{n \to \infty} \sum_{j=1}^m j \sigma_j^n \right) \leq 2r.\end{align*}
Since $\threeN \mathbf{y} \threeN < \infty$, we can conclude the following.

If $y^{(2)} = \sum_{j=1}^\infty \sigma_j X_j \wedge Y_j$ is defined with all vectors orthonormal, we define $F_m = \spn \{ y, X_1, Y_1, \dots, X_m, Y_m\}$. Then
$$\lim_{m\to \infty} \rho( \pr_{F_m} \mathbf{y}, \mathbf{y}) \leq 3 \lim_{m\to \infty} \threeN \mathbf{y} - \pr_{F_m} \mathbf{y} \threeN =0.$$
Using that $\rho(0, \pr_{F_m} \mathbf{y}) \leq \rho(0, \mathbf{y}) \leq \rho(0, \pr_{F_m} \mathbf{y}) + \rho( \pr_{F_m} \mathbf{y}, \mathbf{y})$, it follows that
$$\lim_{m \to \infty} \rho(0, \pr_{F_m} \mathbf{y}) = \rho(0, \mathbf{y}).$$
Furthermore, since
$$\| \pr_{F_m} (\mathbf{y} - \mathbf{y}^n) \|_{\Sch^\infty} \leq \| \mathbf{y} - \mathbf{y}^n \|_{\Sch^\infty} \leq \| \mathbf{y} - \mathbf{y}^n \|_{\otimes},$$
we have that $\lim_{n \to \infty}\| \pr_{F_m} (\mathbf{y} - \mathbf{y}^n) \|_{\Sch^\infty} = 0$. Since all left invariant homogeneous norms are equivalent on a finite dimensional space,
we have convergence $\lim_{n\to \infty} \rho( \pr_{F_m} \mathbf{y}^n , \pr_{F_m} \mathbf{y} ) \to 0$ for any fixed $m$. Finally
$$\rho(0, \mathbf{y} ) = \lim_{m\to \infty} \rho(0, \pr_{F_m} \mathbf{y}) = \lim_{m\to \infty} \lim_{n\to \infty}\rho(0, \pr_{F_m} \mathbf{y}^n) \leq r,$$
so $\mathbf{y} \in \bar{B}_\rho(0, r)$.
\end{proof}

\begin{proof}[Proof of Theorem~\ref{th:main}, Part II] We are now ready to complete the proof.

\

\paragraph{\it Step 5: Every point has a midpoint}
Let $\mathbf{x} = x + x^{(2)} = x + \sum_{j=1}^\infty \sigma_j X_j \wedge Y_j \in G^2_{cc}(E)$ be arbitrary and define $F_n$ as in \eqref{Fn}. 
If we write $\mathbf{x}^n = \pr_{F_n} \mathbf{x}$, then by the definition in \eqref{Kset}, we have that $K(\bx^n) \subseteq K(\bx)$. Since $\mathbf{x}^n \in G_a^2(E)$, there exists a length minimizing geodesic $\bGamma^n$ from $0$ to $\mathbf{x}^n$, which we know is in $K(\bx)$ by Lemma~\ref{lemma:Compact}.

Let $\mathbf{s}^n$ denote the midpoint of each geodesic $\bGamma^n$. This satisfies
$$\rho(0, \mathbf{s}^n) = \rho(\mathbf{x}^n, \mathbf{s}^n) = \frac{1}{2} \rho(0, \mathbf{x}^n) \leq \frac{1}{2} \rho(0, \mathbf{x})  := r. $$
Write $\delta_m = \rho(\mathbf{x}^m, \mathbf{x})$, and define balls
\begin{align*}
\bar{B}_0 & = \{ \mathbf{y} \in G^2(E) \, : \,  \rho (0, \mathbf{y}) \leq r\}, \\
\bar{B}_{m} & = \{ \mathbf{y} \in G^2(E) \, : \, \rho (\mathbf{x}, \mathbf{y}) \leq r+ \delta_m \}.
\end{align*}
By the definition of $\mathbf{x}^n$, we have $\mathbf{s}^n \in \bar{B}_0 \cap \bar{B}_{m}$ for any $n \geq m$ with $\delta_m \to 0$.

Since every $\mathbf{s}^m$ is contained in $K(\mathbf{x})$, by compactness, there is a subsequence $\mathbf{s}_{n_k}$ converging to a point $\mathbf{s}$ in $G^2(E)$. This element hence has to be contained in $\bar{B}_{0} \cap \bar{B}_{m}$ for any $m \geq 1$ by Lemma~\ref{lemma:closed}. It follows that
$$\rho(0, \mathbf{s}) = \rho(\mathbf{x}, \mathbf{s}) = \frac{1}{2} \rho(0, \mathbf{x}),$$
i.e., $\mathbf{s}$ is a midpoint of $\mathbf{x}$. Since $(G_{cc}^2(E), \rho)$ is a complete length space, it follows from left-invariance of the metric together with \cite[Theorem 2.4.16]{BaBaI01} that existence of such midpoint for any element is equivalent to the space being a geodesic space. This completes the proof.
\end{proof}

\subsection{Proof of Theorem~\ref{th:mainHilbert}} \label{sec:ProofHilbert}
We now come to the proof of our main result. Namely, if $E$ is a Hilbert space and we define $\alpha$-weak geometric rough path relative to the tensor norm $\|\cdot\|_{\Sch^p}$, $1 \leq p \leq \infty$ on the tensor product, then for $\beta \in (1/3, \alpha)$
$$
\mathscr{C}_g^{\alpha}([0,T], E) \subset \mathscr{C}^{\alpha}_{wg}([0,T],E)  \subset \mathscr{C}_g^{\beta}([0,T], E)  .
$$
We can prove this by showing that the conditions (I), (II) and (III) in Theorem~\ref{thm:geometric rough paths} are satisfied. By Theorem~\ref{th:main} it follows that (I) and (II) are satisfied for Hilbert spaces. Hence, we only need to prove that condition (III) holds.

Recall the results of Lemma~\ref{lemma:ProjSch}. Let $\bx = x+ x^{(2)} \in C^\alpha([0,T], G^2(E))$ be an arbitrary weakly geometric $\alpha$-rough path. For any fixed $t$, define a sequence of increasing finite subspaces $\{ F_{t, n}\}_{n=1}^\infty$, such that
$$x_t \in F_{t,n} \qquad d(\bx_t, \pr_{F_{t,n} }\bx_t) = d(0, \pr_{F_{t,n}^\perp}  x^{(2)})\leq \frac{1}{n}.$$
Consider a partition $\Pi = \{ t_0 = 0 < t_1 < t_2 < \cdots < t_k = T \}$ of the interval $[0, T]$. Write
$$F_{\Pi,n} = \spn \{ F_{t,n} \, : \, t \in \Pi \}.$$
Define $\bx_t^{\Pi,n} = \pr_{F_{\Pi,n}} \bx_t$. Since $\rho$ and $d$ are equivalent on the finite dimensional $F_{\Pi,n}$, it follows that $\bx_t^{\Pi,n}$ is a continuous function in $G^2_{cc}(E)$ with respect to~$\rho$.

Since $\bx_t$ is uniformly continuous, we can for each $r > 0$ find a number $o_r$ such that
$$o_{r} = \mathrm{Osc}(\bx_t; r) = \sup_{\begin{subarray}{c} 0\leq s < t \leq T \\ t-s \leq r \end{subarray}} d(\bx_s, \bx_t),$$
with $o_r$ approaching $0$ as $r \to 0$. We now see that for every $t \in [t_i, t_{i+1}]$,
\begin{align*}
& d(\bx_t^{\Pi,n}, \bx_t) \leq d(\bx_t^{\Pi,n}, \bx_{t_i}^{\Pi,n}) + d(\bx_{t_i}^{\Pi,n}, \bx_{t_i}) + d(\bx_{t_i}, \bx_t) \\
& \leq 2 d(\bx_t, \bx_{t_i}) + d(0, \pr_{F_{\Pi,n}^\perp} x^{(2)}_{t_i}) \\
& \leq 2 d(\bx_t, \bx_{t_i}) + d(0, \pr_{F_{t_i,n}^\perp} x^{(2)}_{t_i}) \leq 2 o_{|\Pi|} + \frac{1}{n}
\end{align*}
Defining $\bx_t^n = \bx_t^{n, \Pi_n}$ where $\Pi_n$ is a partition with $|\Pi_n| = \frac{1}{n}$, we have that $d(\bx_t^n, \bx_t)$ converges uniformly to $0$.
Using now the $\alpha$-H\"older property of $\mathbf{x}^n$ and $\mathbf{x}$ and an interpolation argument as in the proof of Theorem~\ref{thm:geometric rough paths}, we obtain that $d_\beta(\bx, \bx^n) \to 0$ for any $\beta \in (\frac{1}{3}, \alpha)$. This completes the proof.

\begin{remark}[Other cross-norms]
As one can see from the proof in Section~\ref{sec:ProofHilbert}, what is needed for our result is the properties of Lemma~\ref{lemma:ProjSch} and Theorem~\ref{th:main}. Hence, for any norm on the tensor product which satisfy these two results, the result in Theorem~\ref{th:mainHilbert} holds.
\end{remark}

\subsection{Generalizing the result to Banach spaces}\label{subsect: Banach}
One of the central tools in our proof for geometric rough paths when $E$ is a Hilbert space, is that we can use orthogonal projections $\Pr_F : E \to F$, which all shorten lengths and hence have norm $1$. Such contractive projections are in general rare in Banach spaces as we have the following characterisation from \cite[Theorem 3.1]{Rand01}.
\begin{theorem}
For a Banach space $E$ with $\text{dim } E \geq 3$, the following statements are equivalent:
\begin{enumerate}[\rm (i)]
 \item $E$ is isometrically isomorphic to a Hilbert space,
 \item every $2$-dimensional subspace of $E$ is the range of a projection of norm $1$,
 \item every subspace of $E$ is the range of a projection of norm $1$.
\end{enumerate}
\end{theorem}

\bibliographystyle{new}

\bibliography{SRBanach_lit}

\end{document}